\documentclass[12pt]{article}
\usepackage[utf8]{inputenc}
\usepackage[margin=2.5cm]{geometry}
\usepackage{amsthm}
\usepackage{enumitem}
\usepackage{amsmath}
\setlist[itemize]{topsep=0pt}
\setlist[enumerate]{topsep=0pt}
\usepackage{amsfonts}
\usepackage{ amssymb }
\usepackage[title]{appendix}
\usepackage{adjustbox}
\usepackage{multirow}
\usepackage{multicol}
\usepackage{stmaryrd}
\usepackage{color}
\usepackage{array}
\usepackage{graphicx}
\usepackage{hyperref}
\usepackage{mathtools}
\usepackage{xcolor}
\usepackage{tikz}
\usetikzlibrary{arrows, arrows.meta}
\usepackage{changepage}
\usepackage{titlefoot}
\usepackage{titlesec}
\titleformat*{\section}{\centering\bfseries\small\MakeUppercase}
\titleformat*{\subsection}{\normalsize\bfseries}
\titlespacing\section{0pt}{6pt plus 4pt minus 2pt}{0pt plus 2pt minus 2pt}
\titlespacing{\subsection}{0pt}{-2pt}{-4pt}
\titlespacing{\subsubsection}{0pt}{-2pt}{-4pt}
\usepackage{thm-restate}
\usepackage[capitalise]{cleveref}

\newtheoremstyle{exampstyle}
  {\topsep} 
  {0pt} 
  {\slshape} 
  {} 
  {\bfseries} 
  {.} 
  {.5em} 
  {} 
\theoremstyle{exampstyle}
\newtheorem{thm}{Theorem}[section]
\newtheorem{cor}[thm]{Corollary}
\newtheorem{lemma}[thm]{Lemma}

\newtheorem{prop}[thm]{Proposition}

\newtheoremstyle{def}
  {\topsep} 
  {0pt} 
  {} 
  {} 
  {\bfseries} 
  {.} 
  {.5em} 
  {} 
\theoremstyle{def}
\newtheorem{defn}[thm]{Definition}
\newtheorem{rmk}[thm]{Remark}
\newtheorem{exmp}[thm]{Example}

\newcommand{\comm}[1]{}
\newcommand{\Z}{\mathbb{Z}}

\newcommand{\e}{\varepsilon}

\newcommand{\myRed}[1]{\textbf{\textcolor{red}{#1}}}
\newcommand{\Mod}[1]{\ (\mathrm{mod}\ #1)}

\title{\large{\uppercase{\textbf{Twisted conjugacy in dihedral Artin groups II: Baumslag Solitar groups}} $\mathrm{BS}(n,n)$}}
\author{Gemma Crowe}
\date{}
\begin{document}

\maketitle
\begin{abstract}
In this second paper we solve the twisted conjugacy problem for even dihedral Artin groups, that is, groups with presentation $G(m) = \langle a,b \mid {}_{m}(a,b) =  {}_{m}(b,a) \rangle$, where $m \geq 2$ is even, and $_{m}(a,b)$ is the word $abab\dots$ of length $m$. Similar to odd dihedral Artin groups, we prove orbit decidability for all subgroups $A \leq \mathrm{Aut}(G(m))$, which then implies that the conjugacy problem is solvable in extensions of even dihedral Artin groups.\\

2020 Mathematics Subject Classification: 20F10, 20F36.

\end{abstract}
\unmarkedfntext{\emph{Keywords}: Twisted conjugacy problem, dihedral Artin groups, orbit decidability, Baumslag-Solitar groups.}

\section{Introduction}
In this second paper, we find a complete solution to the following decision problem.

\begin{restatable*}{thm}{TCP}\label{thm:main result TCP solvable}
    The twisted conjugacy problem is decidable for dihedral Artin groups. 
\end{restatable*}
For a finitely generated group $G$ with generating set $X$, we say two elements $u,v \in G$ are \emph{twisted conjugate} by some automorphism $\phi \in \mathrm{Aut}(G)$ if there exists an element $w \in G$ such that $v = \phi^{-1}(w)uw$. The \emph{twisted conjugacy problem} (TCP) asks whether there exists an algorithm to determine if two elements, given as words over $X$, are twisted conjugate by some automorphism $\phi \in \mathrm{Aut}(G)$. In \cite{crowe_twisted_2024}, an implementable algorithm was constructed to solve the twisted conjugacy problem for odd dihedral Artin groups. It remains therefore to solve the twisted conjugacy problem for even dihedral Artin groups, that is, groups with presentation $G(m) = \langle a,b \mid {}_{m}(a,b) =  {}_{m}(b,a) \rangle$, where $m \geq 2$ is even, and $_{m}(a,b)$ is the word $abab\dots$ of length $m$. This is the goal of the present paper. 

We take a similar approach as odd dihedral Artin groups, by considering alternative group presentations for even dihedral Artin groups. First, we observe that any even dihedral Artin group $G(m)$ is isomorphic to a \emph{Baumslag Solitar group} of the form $\mathrm{BS}(n,n)$. With this presentation, \cite{gilbert_tree_2000} describes all outer automorphisms of this group which, unlike for odd dihedral Artin groups, include outer automorphisms which are non-length preserving. Our algorithm to solve the twisted conjugacy problem works for all outer automorphisms, and so we extend Juh\'{a}sz's result \cite{juhasz_twisted_2011}, where only length preserving automorphisms were considered. Secondly, we use the fact that $\mathrm{BS}(n,n)$ is isomorphic to the semidirect product $F_{n} \rtimes \Z$. This semidirect form will allow us to construct an algorithm to solve the twisted conjugacy problem for all even dihedral Artin groups. We believe this is the first example of a Baumslag Solitar group with solvable twisted conjugacy problem. 

As well as establishing \cref{thm:main result TCP solvable}, we are able to determine the complexity of our algorithm for even dihedral Artin groups. We found the complexity of the twisted conjugacy problem for odd dihedral Artin groups to be linear in \cite{crowe_twisted_2024}.
\begin{restatable*}{thm}{Complex}\label{thm:complexity}
    Let $G(m) = \langle X \rangle$ be a dihedral Artin group. Based on the length of input words $u, v \in X^{\ast}$, the twisted conjugacy problem for $G(m)$ has:
    \begin{itemize}
        \item[(i)] linear complexity when $m$ is odd, and
        \item[(ii)] quadratic complexity when $m$ is even.
    \end{itemize}
\end{restatable*}
Similar to our first paper \cite{crowe_twisted_2024}, we extend \cref{thm:main result TCP solvable} and consider the conjugacy problem in extensions of even dihedral Artin groups. The criteria from \cite[Theorem 3.1]{bogopolski_orbit_2009} allows us to determine the conjugacy problem in group extensions, based on a property known as \emph{orbit decidability.} For a group $G$ and subgroup $A \leq \mathrm{Aut}(G)$, the \emph{orbit decidability problem} asks whether we can determine if for two elements $u,v \in G$, there exists an automorphism $\phi \in A$ such that $v$ is conjugate to $\phi(u)$. To apply \cref{thm:main result TCP solvable}, we require what is known as the `action subgroup' of dihedral Artin groups to be orbit decidable. For odd dihedral Artin groups, we proved a stronger statement, by showing that all subgroups of the automorphism group of odd dihedral Artin groups are orbit decidable \cite{crowe_twisted_2024}. We can also prove the same result for even dihedral Artin groups, and so we prove the following theorem.

\begin{restatable*}{thm}{orbit}\label{thm:orbit decid}
    Every finitely generated subgroup $A \leq \mathrm{Aut}(G(m))$, is orbit decidable.
\end{restatable*}
Combined with \cref{thm:main result TCP solvable} and \cite{crowe_twisted_2024}, we find new examples of groups with solvable conjugacy problem. 

\begin{restatable*}{thm}{extension}
    Let $G = G(m) \rtimes H$ be an extension of a dihedral Artin group by a finitely generated group $H$ which satisfies conditions $(ii)$ and $(iii)$ from \cite[Theorem 3.1]{bogopolski_orbit_2009} (e.g. let $H$ be torsion-free hyperbolic). Then $G$ has decidable conjugacy problem. 
\end{restatable*}
The structure of this paper is as follows. In \cref{sec:prelims}, we provide details on group presentations and the outer automorphism group for even dihedral Artin groups. We then use algebraic properties of the semidirect presentation $F_{n} \rtimes \Z$ in \cref{sec:algorithm} to construct an algorithm to solve the twisted conjugacy problem in even dihedral Artin groups. The main idea behind our algorithm is that any two twisted conjugate elements, which are minimal length within their twisted conjugacy class, must be related by a sequence of `shifts'. Finally in \cref{sec:extension}, we show orbit decidability for all subgroups of $\mathrm{Aut}(G(m))$, and apply our results to solve the conjugacy problem in extensions of dihedral Artin groups. 

\section{Preliminaries}\label{sec:prelims}
All relevant background knowledge on decision problems and dihedral Artin groups can be found in \cite{crowe_twisted_2024}. We recall necessary notation and definitions here.

Let $X$ be a finite set, and let $X^{\ast}$ be the set of all finite words over $X$. For a group $G$ generated by $X$, we use $u=v$ to denote equality of words in $X^{\ast}$, and $u =_{G} v$ to denote equality of the group elements represented by $u$ and $v$. We let $l(w)$ denote the word length of $w$ over $X$. For a group element $g \in G$, we define the \emph{length} of $g$, denoted $|g|_{X}$, to be the length of a shortest representative word for the element $g$ over $X$. A word $w \in X^{\ast}$ is \emph{geodesic} if $l(w) = |\pi(w)|_{X}$, where $\pi: X^{\ast} \rightarrow G$ is the natural projection. If there exists a unique word $w$ of minimal length representing $g$, then we say $w$ is a \emph{unique geodesic}. Otherwise, $w$ is a non-unique geodesic. We write $u \sim v$ when $u,v \in X^{\ast}$ represent conjugate elements in $G$.

\begin{defn}
    Let $G = \langle X \rangle$, let $u,v \in X^{\ast}$, and let $\phi \in \mathrm{Aut}(G)$ be an automorphism of $G$.
    \begin{enumerate}
        \item We say $u$ and $v$ are \emph{$\phi$-twisted conjugate}, denoted $u \sim_{\phi} v$, if there exists an element $w \in G$ such that $v =_{G} \phi(w)^{-1}uw$.
        \item The \emph{$\phi$-twisted conjugacy problem} for $G$, denoted $\mathrm{TCP}_{\phi}(G)$, takes as input two words $u,v \in X^{\ast}$, and decides whether they represent groups elements which are $\phi$-twisted conjugate to each other in $G$.
        \item The (uniform) \emph{twisted conjugacy problem} for $G$, denoted $\mathrm{TCP}(G)$, takes as input two words $u,v \in X^{\ast}$, and decides whether $u$ and $v$ represent groups elements which are $\phi$-twisted conjugate in $G$, for some $\phi \in \mathrm{Aut}(G)$.
    \end{enumerate}
\end{defn}
\begin{defn}\label{defn: gen set free product}
    Let $m \in \Z_{\geq 2}$. A dihedral Artin group is the group defined by the following presentation:
    \begin{equation}\label{eq:all groups}
    G(m) = \langle x,y \mid {}_{m}(x,y) =  {}_{m}(y,x) \rangle,
\end{equation}
     where $_{m}(x,y)$ is the word $xyxy \dots$ of length $m$. When $m$ is even, this group is isomorphic to
    \begin{equation}\label{eq: even free product present}
    G(m) = \langle a,b \mid a^{-1}b^{n}a = b^{n} \rangle,
\end{equation}
where $n = \frac{m}{2} \geq 2$, by setting $a=x, b=xy$. Here $G(m)$ is isomorphic to the \emph{Baumslag-Solitar group} $\mathrm{BS}(n,n)$.
\end{defn}

\subsection{Outer automorphisms}
Similar to odd dihedral Artin groups, we establish the behaviour of the outer automorphism group of $G(m)$. These were classified in $\cite{gilbert_tree_2000}$.

\begin{thm}\cite[Theorem D]{gilbert_tree_2000}
    Let $G(m)$ be a group with presentation given by \cref{eq: even free product present}. Then
    \[     \mathrm{Out}(G(m)) \cong D_{\infty} \times C_{2}.
    \]
\end{thm}
We define the following three automorphisms of $G(m)$:
\begin{equation*}
    \begin{split}
    \beta_{1}(a) &= a, \\
    \beta_{2}(a) &= a^{-1}, \\
    \beta_{3}(a) &= ab,
\end{split}
\quad 
\begin{split}
    \beta_{1}(b) &= b^{-1},  \\
    \beta_{2}(b) &= b,  \\
    \beta_{3}(b) &= b.
\end{split}
\end{equation*}
\begin{prop}
    The images of $\beta_{1}, \beta_{2}, \beta_{3}$ generate $\mathrm{Out}(G(m)) = D_{\infty} \times C_{2}$.
\end{prop}

\begin{proof}
    We will show that $\beta_{1}$ generates $C_{2}$, and that $\{\beta_{2}, \beta_{3}\}$ generates $D_{\infty}$. We check the following relations in $\mathrm{Out}(G(m))$:
    \[ \beta^{2}_{1} = \beta^{2}_{2} = (\beta_{1}\beta_{2})^{2} = (\beta_{1}\beta_{3})^{2} = (\beta_{2}\beta_{3})^{2} = 1.
    \]
    It is clear that $\beta_{1}$ and $\beta_{2}$ have order two in both $\mathrm{Out}(G(m))$ and $\mathrm{Aut}(G(m))$. Consider the following compositions:
    \begin{equation*}
        \begin{split}
            \beta_{1}\beta_{2} &\colon a \mapsto a^{-1}, \\
            \beta_{1}\beta_{3} &\colon a \mapsto ab^{-1}, \\
            \beta_{2}\beta_{3} &\colon a \mapsto a^{-1}b,
        \end{split}
        \quad 
        \begin{split}
            b &\mapsto b^{-1}, \\
            b &\mapsto b^{-1}, \\
            b &\mapsto b. 
        \end{split}
    \end{equation*}
    Then $(\beta_{1}\beta_{2})^{2}$ and $(\beta_{1}\beta_{3})^{2}$ have order two in both $\mathrm{Out}(G(m))$ and $\mathrm{Aut}(G(m))$. Also consider
    \[ (\beta_{2}\beta_{3})^{2}\colon a \mapsto b^{-1}ab, \; b \mapsto b,
    \]
    and so $(\beta_{2}\beta_{3})^{2}$ has order 2 in $\mathrm{Out}(G(m))$. With these relations, we obtain the presentation required:
    \[ \mathrm{Out}(G(m)) \cong D_{\infty} \times C_{2} = \langle \beta_{1}, \beta_{2}, \beta_{3} \mid \beta_{1}^{2} = \beta_{2}^{2} = (\beta_{1}\beta_{2})^{2} = (\beta_{1}\beta_{3})^{2} = (\beta_{2}\beta_{3})^{2} = 1 \rangle. 
    \]
\end{proof}
\begin{cor}\label{cor: odd outer old}
    Any element $g \in \mathrm{Out}(G(m))$ can be written in the form $g = \beta_{1}^{q} \beta_{2}^{r} \beta_{3}^{s}$, where $q,r \in \{-1,0,1\}, s \in \mathbb{Z}$. In particular, any $\phi \in \mathrm{Out}(G(m))$ is of the form
    \[ \phi \colon a \mapsto a^{\e_{a}}b^{d}, \; b \mapsto b^{\e_{b}},
    \]
    where $\e_{a}, \e_{b} \in \{\pm 1\},$ and $d \in \Z$.
\end{cor}
\comm{
In particular, any $g \in \mathrm{Out}(G(m))$ maps $ a \mapsto a^{\pm 1}b^{d}, \; b \mapsto b ^{\pm 1}$, where $d \in \mathbb{Z}$. We can now classify all finite order elements of $\mathrm{Out}(G(m))$.
\begin{cor}
    Any $\phi \in \mathrm{Out}(G(m))$ of finite order is of the form
    \[ \phi: a \mapsto a^{\e_{a}}b^{d}, \; b \mapsto b^{\e_{b}},
    \]
    where $\e_{a}, \e_{b} \in \{\pm 1\}, d \in \Z$, and if $d \neq 0$, then $\e_{a} \neq \e_{b}$. 
\end{cor}
\begin{proof}
    We consider length preserving and non-length preserving cases in turn. Indeed when $d = 0$, it is clear than $\phi$ must be of order 2. When $d \neq 0$, we note that the only maps with infinite order in $\mathrm{Out}(G(m))$ are those of the form
    \begin{align*}
         \phi_{1}&: a \mapsto ab^{d}, \; b \mapsto b,\\
         \phi_{2}&: a \mapsto a^{-1}b^{d}, \; b \mapsto b^{-1}.
    \end{align*}
    Therefore $\e_{a} \neq \e_{b}$ when $d \neq 0$.
\end{proof}}
\subsection{Semidirect product presentation}\label{sec:semidirect pres}
To solve the twisted conjugacy problem $\mathrm{TCP}(G(m))$, we use an alternative presentation for $G(m)$.
\begin{prop}\cite[Prop. 5.1]{felshtyn_twisted_2018}\label{prop:semidirect pres}
    The group $G(m)$ is isomorphic to the following semidirect product:
    \begin{equation}\label{eq: even semidirect pres}
        G(m) \cong F_{n} \rtimes \Z = \langle x_{0}, x_{1}, \dots,x_{n-1}, y \mid y^{-1}x_{i}y = x_{i+1} \; (1 \leq i \leq n-2), \; y^{-1}x_{n-1}y = x_{0} \rangle,
    \end{equation}
    where $n = \frac{m}{2} \geq 2$. The isomorphism is defined as 
\begin{align*}
        \psi \colon G(m) &\rightarrow F_{n} \rtimes \Z \\
        a &\mapsto x_{0}, \\
        b &\mapsto y. 
    \end{align*}
\end{prop}
We note that if $u,v \in F_{n}$ are twisted conjugate by an element $w \in F_{n}$, then we can solve the twisted conjugacy problem with respect to $u$ and $v$, since the twisted conjugacy problem is solvable in free groups \cite[Theorem 1.5]{bogopolski_conjugacy_2006}.

For notation, we let $X_{n} = \{x_{0}, \dots, x_{n-1}\}$, and $X = X_{n} \cup \{y\}$. \cref{prop:semidirect pres} implies that any geodesic $g \in X^{\ast}$ can be written uniquely in the form $g =_{G} g_{1}g_{2}$, where $g_{1} \in F_{n}$ and $g_{2} = y^{t}$ for some $t \in \Z$. We call this our \emph{geodesic normal form}. We will refer to $g_{1} \in F_{n}$ as the \emph{free component} of $g$, and the set $X_{n}$ as the \emph{free generators}. 
\begin{rmk}
    Unless otherwise stated, any (geodesic) elements $u \in F_{n}$ will be written in the form $u = x_{i_{1}}^{p_{1}}\dots x^{p_{q}}_{i_{q}},$ where $p_{j} \in \{\pm 1\}$ and $i_{j} \in \{0, \dots, n-1\} \; (1 \leq j \leq q)$. 
\end{rmk}
Let $[ \cdot ] \colon \Z \rightarrow \{0, \dots, n-1 \}$ denote the modular function, which maps any integer $z$ to $z \Mod{n}$. Using the relations from \cref{prop:semidirect pres}, we can show the following. 
\begin{lemma}\label{lem:relation x y}
    Let $G(m)$ be defined by the presentation given in \cref{prop:semidirect pres}. Then for all $s \in \Z_{\neq 0}, \; i \in \{0, \dots, n-1\}, \; \e = \pm 1$, we have 
    \[ y^{s}x^{\e}_{i} =_{G} x^{\e}_{[i-s]}y^{s}.
    \]
\end{lemma}

\begin{proof}
    We prove by induction on $s \in \Z_{> 0}$, and the proof is symmetric for $s \in \Z_{<0}$. The base case is immediate from the relations in $G(m)$. By the inductive hypothesis, we have
    \[ y^{s+1}x^{\e}_{i} =_{G} yx^{\e}_{[i-s]}y^{s} =_{G} x^{\e}_{[i-(s+1)]}y^{s+1},     \]
    as required.
\end{proof}
This result illustrates that where we move $y$ powers to the right hand side of a word over $X$, then we only need to change the index of the free generators. We now define a function which captures this change of index of free generators. 
\begin{defn}\label{defn:map move y right}
    For all $s \in \Z_{\neq 0}$, define
\begin{align*}
        \Phi_{s} \colon F_{n} &\rightarrow F_{n} \\
        x^{\e}_{i} &\mapsto x^{\e}_{[i-s]},
    \end{align*}
where $\e \in \{ \pm 1 \}, i \in \{0, \dots, n-1\}$. This function is a bijection and preserves word length. Since $\Phi_{s}\circ \Phi_{t}\left(x^{\e}_{i}\right) = \Phi_{s+t}\left(x^{\e}_{i}\right)$ for all $s, t \in \Z_{\neq 0}$, it is an automorphism of $F_{n}$. In particular, $\Phi^{-1}_{s} = \Phi_{-s}$.
\end{defn}
Finally, we rewrite our outer automorphisms with respect to our new presentation as in \cref{eq: even semidirect pres}. By \cref{cor: odd outer old} and the isomorphism defined in \cref{prop:semidirect pres}, any outer automorphism of $G(m)$ is of the form
\begin{equation}\label{even:all auto forms}
    \phi \colon x_{0} \mapsto x_{0}^{\e_{x}}y^{d}, \; y \mapsto y^{\e_{y}},
\end{equation}
where $\e_{x}, \e_{y} \in \{\pm 1\},$ and $d \in \Z$. Note $\e^{2}_{x} = \e^{2}_{y} = 1$. For the remainder of this paper, $\phi \in \mathrm{Out}(G(m))$ will be of the form as in \cref{even:all auto forms}.

\subsection{Algebraic results}
Our first goal is to understand the image of geodesics by $\phi \in \mathrm{Out}(G(m))$, as in \cref{even:all auto forms}. The following result can be shown using the relations from the presentation as in \cref{eq: even semidirect pres}, and \cref{lem:relation x y}.
\begin{lemma}\label{lem:defn phi image gens}
    Let $0 \leq i \leq n-1$, and let $r_{i} \in \{\pm 1\}$. For all free generators $x_{i} \in X_{n}$, we have
    \[ \phi\left(x^{r_{i}}_{i}\right) =_{G} 
    \begin{cases}
        x^{\e_{x}r_{i}}_{[\e_{y}i]}y^{r_{i}d}, & r_{i} = 1, \\
        x^{\e_{x}r_{i}}_{[\e_{y}i+d]}y^{r_{i}d}, & r_{i} = -1. 
    \end{cases}
    \]
\end{lemma}
For notation, we define the following function which describes the free component of the image of a free generator $x_{i} \in X_{n}$ by $\phi$.
\begin{defn}\label{defn:phi map free comp}
    Let $0 \leq i \leq n-1$, and let $r_{i} \in \{\pm 1\}$. For all free generators $x_{i} \in X_{n}$, define
    \begin{align*}
    \phi_{F} \colon F_{n} &\rightarrow F_{n} \\
    x^{r_{i}}_{i} &\mapsto  
    \begin{cases}
        x^{\e_{x}}_{[\e_{y}i]}, & r_{i} = 1, \\
        x^{-\e_{x}}_{[\e_{y}i+d]}, & r_{i} = -1. 
    \end{cases}
\end{align*}
    In particular, $\phi(x^{r_{i}}_{i}) =_{G} \phi_{F}(x^{r_{i}}_{i})y^{r_{i}d}$.
\end{defn}
Note $\phi_{F}$ is not a homomorphism, for example $\phi_{F}(x^{r_{i}}_{i})^{-1} \neq \phi_{F}(x^{-r_{i}}_{i})$. We now define equivalent functions for the inverse map of $\phi$.  

\begin{lemma}\label{defn:inverse map}
    Let $0 \leq i \leq n-1$, and let $r_{i} \in \{\pm 1\}$. For all free generators $x_{i} \in X_{n}$, we have
    \[ \phi^{-1}(x^{r_{i}}_{i}) =_{G} 
    \begin{cases}
        x^{\e_{x}r_{i}}_{[\e_{y}i]}y^{-\e_{x}\e_{y}r_{i}d}, & \e_{x}r_{i} = 1, \\
        x^{\e_{x}r_{i}}_{[\e_{y}(i-d)]}y^{-\e_{x}\e_{y}r_{i}d}, & \e_{x}r_{i} = -1.
    \end{cases}
    \]
    Moreover, we define a function which describes the free component of the image of a free generator $x_{i} \in X_{n}$ by $\phi^{-1}$:
    \begin{align*}
        \phi^{-1}_{F} \colon F_{n} &\rightarrow F_{n} \\
        x^{r_{i}}_{i} &\mapsto 
        \begin{cases}
            x_{[\e_{y}i]}, & \e_{x}r_{i} = 1, \\
            x^{-1}_{[\e_{y}(i-d)]}, & \e_{x}r_{i} = -1.
        \end{cases}
    \end{align*}
\end{lemma}
\begin{proof}
    If $\e_{x}r_{i} = 1$, then $\phi\left(x^{\e_{x}r_{i}}_{[\e_{y}i]}y^{-\e_{x}\e_{y}r_{i}d}\right) =_{G} x^{r_{i}}_{i}y^{\e_{x}r_{i}d}y^{-\e_{x}r_{i}d} =_{G} x^{r_{i}}_{i}$. A similar proof holds when $\e_{x}r_{i} = -1$. 
\end{proof}

\begin{exmp}\label{exmp:first}
    Let $n=3$, i.e. $F_{n} = \langle x_{0}, x_{1}, x_{2} \rangle$. Let $\phi \colon x_{0} \mapsto x_{0}y^{4}, \; y \mapsto y$, i.e. $\e_{x} = \e_{y} =1$ and $d = 4$. Consider the geodesic $u = x_{0}x_{2}y^{2} \in F_{n} \rtimes \Z$. We can compute $\phi(u)$ as follows:
    \[ \phi(u) =_{G} \phi_{F}(x_{0})y^{4}\phi_{F}(x_{2})y^{4}y^{2} = x_{0}y^{4}x_{2}y^{6} =_{G} x_{0}\Phi_{4}(x_{2})y^{10} = x_{0}x_{1}y^{10}.
    \]
\end{exmp}
We now prove some algebraic results related to these functions, which will be used later. 
\begin{prop}\label{prop:non len p 1}
    Let $s \in \Z_{\neq 0}$, $0 \leq i \leq n-1$ and $r_{i} \in \{\pm 1\}$. For all free generators $x_{i} \in X_{n}$,
    \begin{enumerate}
        \item[(i)] $\Phi_{s}\left(\phi^{\pm 1}_{F}\left(x^{r_{i}}_{i}\right)\right) = \phi^{\pm 1}_{F}\left(\Phi_{\e_{y}s}\left(x^{r_{i}}_{i}\right)\right)$.
        \item[(ii)] $\phi^{-1}\left(\phi_{F}\left(x^{r_{i}}_{i}\right)\right) =_{G} x^{r_{i}}_{i}y^{-\e_{y}r_{i}d}$.
        \item[(iii)] $\phi\left(\phi^{-1}_{F}\left(x^{r_{i}}_{i}\right)\right) =_{G} x^{r_{i}}_{i}y^{\e_{x}r_{i}d}.$
    \end{enumerate}
\end{prop}

\begin{proof}
    We provide details for $\phi \in \mathrm{Out}(G(m))$, and the relations for $\phi^{-1}$ follows similarly by \cref{defn:inverse map}. By \cref{defn:phi map free comp}, if $r_{i} = 1$, then $(i)$ follows  from the following relations (recall $\e_{y}^{2} = 1$):
\[ \Phi_{s}\left(\phi_{F}\left(x_{i}\right)\right) = \Phi_{s}\left(x^{\e_{x}}_{[\e_{y}i]}\right) = x^{\e_{x}}_{[\e_{y}i - s]}, \quad \phi_{F}\left(\Phi_{\e_{y}s}\left(x_{i}\right)\right) = \phi_{F}\left(x_{[i-\e_{y}s]}\right) = x^{\e_{x}}_{[\e_{y}i - s]}.
\]
    Similarly if $r_{i} = -1$, then we have:
    \[ \Phi_{s}\left(\phi_{F}\left(x^{-1}_{i}\right)\right) = \Phi_{s}\left(x^{-\e_{x}}_{[\e_{y}i+d]}\right) = x^{-\e_{x}}_{[\e_{y}i - s + d]}, \quad \phi_{F}\left(\Phi_{\e_{y}s}\left(x^{-1}_{i}\right)\right) = \phi_{F}\left(x^{-1}_{[i-\e_{y}s]}\right) = x^{-\e_{x}}_{[\e_{y}i-s+d]}.
    \]
    For $(ii)$ and $(iii)$, we note that
    \[ \phi\left(x_{i}^{r_{i}}y^{-\e_{y}r_{i}d}\right) =_{G} \phi_{F}\left(x^{r_{i}}_{i}\right)y^{r_{i}d}y^{-r_{i}d} =_{G} \phi_{F}\left(x^{r_{i}}_{i}\right), 
    \]
    and so (ii) follows by taking $\phi^{-1}$ of both sides. Similarly for (iii) we have
    \[ \phi^{-1}\left(x^{r_{i}}_{i}y^{\e_{x}r_{i}d}\right) =_{G} \phi^{-1}_{F}(x^{r_{i}}_{i})y^{-\e_{x}\e_{y}r_{i}d}y^{\e_{x}\e_{y}r_{i}d} =_{G} \phi^{-1}_{F}\left(x^{r_{i}}_{i}\right).   \]
\end{proof}

\begin{defn}\label{defn: free component geos}
    For any $v \in F_{n} \rtimes \Z$ in geodesic normal form, let $[v]_{F}$ denote the free component of $v$, that is, $[v]_{F} \in F_{n}$. 
\end{defn}

\begin{prop}\label{prop:non-len p results 2}
    Let $s \in \Z_{\neq 0}$, $w = x^{t_{1}}_{k_{1}}\dots x^{t_{z}}_{k_{z}} \in F_{n}$, and $\sigma = \sum_{i=1}^{z} t_{i}$. Then
    \begin{enumerate}
        \item[(i)] $[\phi(w)]_{F} =_{G} \phi_{F}\left(x^{t_{1}}_{k_{1}}\right)\Phi_{t_{1}d}\left(\phi_{F}\left(x^{t_{2}}_{k_{2}}\right)\right)\Phi_{(t_{1}+t_{2})d}\left(\phi_{F}\left(x^{t_{3}}_{k_{3}}\right)\right)\dots \Phi_{(t_{1}+ \dots + t_{z-1})d}\left(\phi_{F}\left(x^{t_{z}}_{k_{z}}\right)\right)$.
        \item[(ii)] $[\phi^{-1}(w)]_{F} =_{G} \phi^{-1}_{F}\left(x^{t_{1}}_{k_{1}}\right)\Phi_{-\e_{x}\e_{y}t_{1}d}\left(\phi^{-1}_{F}\left(x^{t_{2}}_{k_{2}}\right)\right)\dots \Phi_{-\e_{x}\e_{y}d(t_{1}+\dots + t_{z -1})}\left(\phi^{-1}_{F}\left(x^{t_{z}}_{k_{z}}\right)\right).$
        \item[(iii)] $\phi\left(\left[\phi^{-1}(w)\right]_{F}\right) =_{G} w\cdot y^{\e_{x}\sigma d}$, and $\phi^{-1}\left(\left[\phi(w)\right]_{F}\right) =_{G} w\cdot y^{-\e_{y}\sigma d}$.
        \item[(iv)] $\Phi_{s}\left(\left[\phi^{\pm 1}(w)\right]_{F}\right) =_{G} \left[\phi^{\pm 1}\left(\Phi_{\e_{y}s}(w)\right)\right]_{F}$.
        \item[(v)] $\left(\left[\phi(w)\right]_{F}\right)^{-1} =_{G} \Phi_{\sigma d}\left(\left[\phi(w)^{-1}\right]_{F}\right)$.
        \item[(vi)] The exponent sum of $[\phi(w)]_{F}$ is equal to $\e_{x}\sigma$.
        \item[(vii)] The exponent sum of $\left[\phi(w)^{-1}\right]_{F}$ is equal to $-\e_{x}\sigma$.  
    \end{enumerate}
\end{prop}

\begin{proof}
    By \cref{lem:defn phi image gens} and \cref{defn:phi map free comp} we have
    \begingroup
    \addtolength{\jot}{0.5em}
    \begin{align*}
        \phi(w) &=_{G} \phi_{F}\left(x^{t_{1}}_{k_{1}}\right)y^{t_{1}d}\phi_{F}\left(x^{t_{2}}_{k_{2}}\right)y^{t_{2}d}\phi_{F}\left(x^{t_{3}}_{k_{3}}\right)\dots \phi_{F}\left(x^{t_{z}}_{k_{z}}\right)y^{t_{z}d} \\
        &=_{G} \phi_{F}\left(x^{t_{1}}_{k_{1}}\right)\Phi_{t_{1}d}\left(\phi_{F}\left(x^{t_{2}}_{k_{2}}\right)\right)\Phi_{(t_{1}+t_{2})d}\left(\phi_{F}\left(x^{t_{3}}_{k_{3}}\right)\right)\dots \Phi_{(t_{1} + \dots + t_{z-1})d}\left(\phi_{F}\left(x^{t_{z}}_{k_{z}}\right)\right)y^{\sigma d},
    \end{align*}
    \endgroup
    after moving $y$ terms to the right, using \cref{defn:map move y right}. Similarly by \cref{defn:inverse map} we have
    \begingroup
    \addtolength{\jot}{0.5em}
    \begin{align*}
        \phi^{-1}(w) &=_{G} \phi^{-1}_{F}\left(x^{t_{1}}_{k_{1}}\right)y^{-\e_{x}\e_{y}t_{1}d}\phi^{-1}_{F}\left(x^{t_{2}}_{k_{2}}\right)y^{-\e_{x}\e_{y}t_{2}d}\dots \phi^{-1}_{F}\left(x^{t_{z}}_{k_{z}}\right)y^{-\e_{x}\e_{y}t_{z}d} \\
        &=_{G} \phi^{-1}_{F}\left(x^{t_{1}}_{k_{1}}\right)\Phi_{-\e_{x}\e_{y}t_{1}d}\left(\phi^{-1}_{F}\left(x^{t_{2}}_{k_{2}}\right)\right)\dots \Phi_{-\e_{x}\e_{y}d(t_{1}+ \dots + t_{z - 1})}\left(\phi^{-1}_{F}\left(x^{t_{z}}_{k_{z}}\right)\right)y^{-\e_{x}\e_{y}\sigma d}. 
    \end{align*}
    \endgroup
    These relations prove $(i)$ and $(ii)$. For $(iii)$, recall that $\e^{2}_{y}=1$, and so we can use $(i)$ from \cref{prop:non len p 1} to rewrite $(ii)$ as
    \[ [\phi^{-1}(w)]_{F} =_{G} \phi^{-1}_{F}\left(x^{t_{1}}_{k_{1}}\right)\phi^{-1}_{F}\left(\Phi_{-\e_{x}t_{1}d}\left(x^{t_{2}}_{k_{2}}\right)\right)\dots \phi^{-1}_{F}\left(\Phi_{-\e_{x}d(t_{1}+ \dots + t_{z - 1})}\left(x^{t_{z}}_{k_{z}}\right)\right).
    \]
    By $(iii)$ of \cref{prop:non len p 1}, we then have
    \begingroup
    \addtolength{\jot}{0.5em}
    \begin{align*}
        \phi\left(\left[\phi^{-1}(w)\right]_{F}\right) &=_{G} x^{t_{1}}_{k_{1}}y^{\e_{x}t_{1}d}\Phi_{-\e_{x}t_{1}d}\left(x^{t_{2}}_{k_{2}}\right)y^{\e_{x}t_{2}d}\dots \Phi_{-\e_{x}d(t_{1}+ \dots + t_{z - 1})}\left(x^{t_{z}}_{k_{z}}\right)y^{\e_{x}t_{z}d} \\
        &=_{G} x^{t_{1}}_{k_{1}}x^{t_{2}}_{k_{2}}\dots x^{t_{z}}_{k_{z}}y^{\e_{x}\sigma d}.
    \end{align*}
    \endgroup
    A similar method proves the second relation for $(iii)$. For $(iv)$, we provide details for the map $\phi$, and the equivalent relation for $\phi^{-1}$ follows a similar proof. For the left hand side, we have
    \[ \Phi_{s}\left(\left[\phi(w)\right]_{F}\right) =_{G} \Phi_{s}\left(\phi_{F}\left(x^{t_{1}}_{k_{1}}\right)\Phi_{t_{1}d}\left(\phi_{F}\left(x^{t_{2}}_{k_{2}}\right)\right)\dots \Phi_{(t_{1}+\dots + t_{z - 1})d}\left(\phi_{F}\left(x^{t_{z}}_{k_{z}}\right)\right)\right).
    \]
    For the right hand side, we have
    \begingroup
    \addtolength{\jot}{0.5em}
    \begin{align*}
        \left[\phi\left(\Phi_{\e_{y}s}(w)\right)\right]_{F} &=_{G} \phi_{F}\left(\Phi_{\e_{y}s}\left(x^{t_{1}}_{k_{1}}\right)\right)\Phi_{t_{1}d}\left(\phi_{F}\left(\Phi_{\e_{y}s}\left(x^{t_{2}}_{k_{2}}\right)\right)\right)\dots \Phi_{(t_{1}+\dots + t_{z - 1})d}\left(\phi_{F}\left(\Phi_{\e_{y}s}\left(x^{t_{z}}_{k_{z}}\right)\right)\right) \\
        &= \Phi_{s}\left(\phi_{F}\left(x^{t_{1}}_{k_{1}}\right)\right)\Phi_{s+t_{1}d}\left(\phi_{F}\left(x^{t_{2}}_{k_{2}}\right)\right)\dots \Phi_{s+(t_{1}+\dots + t_{z - 1})d}\left(\phi_{F}\left(x^{t_{z}}_{k_{z}}\right)\right) \\
        &= \Phi_{s}\left(\phi_{F}\left(x^{t_{1}}_{k_{1}}\right)\Phi_{t_{1}d}\left(\phi_{F}\left(x^{t_{2}}_{k_{2}}\right)\right)\dots \Phi_{(t_{1}+\dots + t_{z - 1})d}\left(\phi_{F}\left(x^{t_{z}}_{k_{z}}\right)\right)\right) \\
        &=_{G} \Phi_{s}([\phi(w)]_{F}),
    \end{align*}
    \endgroup
    using $(i)$ of \cref{prop:non len p 1}. To prove $(v)$, we consider each side of the equation in turn. The 
    left hand side is equal to
    \[ \left(\left[\phi(w)\right]_{F}\right)^{-1} =_{G} \Phi_{(t_{1}+t_{2} + \dots + t_{z-1})d}\left(\phi_{F}\left(x^{t_{z}}_{k_{z}}\right)^{-1}\right)\dots \phi_{F}\left(x^{t_{1}}_{k_{1}}\right)^{-1}.
    \]
    For the right hand side, we first consider $\phi(w)^{-1}$, which gives us
    \begin{align*}
        \phi(w)^{-1} &=_{G} y^{-\sigma d} \Phi_{(t_{1}+\dots + t_{z-1})d}\left(\phi_{F}\left(x^{t_{z}}_{k_{z}}\right)^{-1}\right)\dots \phi_{F}\left(x^{t_{1}}_{k_{1}}\right)^{-1} \\
        &=_{G} \Phi_{-t_{z}d}\left(\phi_{F}\left(x^{t_{z}}_{k_{z}}\right)^{-1}\right)\dots \Phi_{-\sigma d}\left(\phi_{F}\left(x^{t_{1}}_{k_{1}}\right)^{-1}\right)y^{-\sigma d}.
    \end{align*}
    Therefore
    \begin{align*}
        \left[\phi(w)^{-1}\right]_{F} &=_{G} \Phi_{-t_{z}d}\left(\phi_{F}\left(x^{t_{z}}_{k_{z}}\right)^{-1}\right)\dots \Phi_{-\sigma d}\left(\phi_{F}\left(x^{t_{1}}_{k_{1}}\right)^{-1}\right) \\
        &= \Phi_{-\sigma d}\left(\Phi_{(t_{1}+t_{2} + \dots + t_{z-1})d}\left(\phi_{F}\left(x^{t_{z}}_{k_{z}}\right)^{-1}\right)\dots \phi_{F}\left(x^{t_{1}}_{k_{1}}\right)^{-1}\right) \\
        &=_{G} \Phi_{-\sigma d}\left(\left(\left[\phi(w)\right]_{F}\right)^{-1}\right),
    \end{align*}
    as required. Finally we note that each free generator $x^{t_{i}}_{k_{i}} \in w$ will contribute $\e_{x}t_{i}$ to the exponent sum of $[\phi(w)]_{F}$. This is then inverted for $\left[\phi(w)^{-1}\right]_{F}$, which proves $(vi)$ and $(vii)$.
\end{proof}

\begin{cor}\label{lemma:1}
    Let $w = w_{1}w_{2} \in F_{n}$ be geodesic, where $w_{1} =  x^{t_{1}}_{k_{1}}\dots x^{t_{z_{1}}}_{k_{z_{1}}}$, and $w_{2} = x^{t_{z_{1}+1}}_{k_{z_{1}+1}}\dots x^{t_{z_{2}}}_{k_{z_{2}}}$, for some $1 \leq z_{1} \leq z_{2}$. Let $\sigma_{1} = \sum_{i=1}^{z_{1}} t_{i}$. Then
    \[ [\phi(w)]_{F} =_{G} [\phi(w_{1})]_{F}\cdot \Phi_{\sigma_{1}d}\left(\left[\phi(w_{2})\right]_{F}\right).
    \]
\end{cor}

\begin{proof}
    The result follows by $(i)$ of \cref{prop:non-len p results 2}.
\end{proof}

\begin{lemma}\label{lemma:3}
    Let $w,v \in F_{n}$. Then $[\phi(w)]_{F} =_{G} v$ if and only if $w =_{G} [\phi^{-1}(v)]_{F}$.      
\end{lemma}
\begin{proof}
    For the forward direction, we can assume that $\phi(w) =_{G} vy^{\sigma d}$, where $\sigma$ equals the exponent sum of $w$. Then $w =_{G} \phi^{-1}\left(vy^{\sigma d}\right)$, and so $\left[\phi^{-1}\left(vy^{\sigma d}\right)\right]_{F} = \left[\phi^{-1}(v)\right]_{F} =_{G} w$. The reverse direction follows a symmetric argument. 
\end{proof}

\section{Algorithm for the Twisted Conjugacy Problem}\label{sec:algorithm}
Our main aim of this section is to solve the twisted conjugacy problem $\mathrm{TCP}_{\phi}(G(m))$, with respect to a given $\phi \in \mathrm{Out}(G(m))$ as in \cref{even:all auto forms}. We will show that for any two words $u,v \in X^{\ast}$ such that $u \sim_{\phi} v$, there exists minimal length representatives $\overline{u}$ and $\overline{v}$, which are twisted conjugate to $u$ and $v$ respectively, such that there exists a finite sequence of `shifts' between $\overline{u}$ and $\overline{v}$. These shifts will act in a similar way as $\phi$-cyclic permutations defined in \cite{crowe_twisted_2024}, which will enable us to construct a finite set of minimal length representatives for each twisted conjugacy class. 

To solve the more general twisted conjugacy problem $\mathrm{TCP}(G(m))$, we cannot check all $\phi \in \mathrm{Out}(G(m))$, since $\mathrm{Out}(G(m))$ is infinite. However, we can establish conditions on the parameter $d \in \Z$, which is then sufficient to determine if $u \sim_{\phi} v$ for some $\phi \in \mathrm{Out}(G(m))$, where $u,v \in X^{\ast}$.

\subsection{Shifts}
For this section, we assume $\phi \in \mathrm{Out}(G(m))$ is known and of the form in \cref{even:all auto forms}.

Our first step is to introduce a new normal form, which will be easier to work with when considering shifts. Recall our geodesic normal form, that is, any geodesic $g \in X^{\ast}$ can be written uniquely in the form $g =_{G} g_{1}g_{2}$, where $g_{1} \in F_{n}$ and $g_{2} = y^{t}$ for some $t \in \Z$. We can write $t = kn + c$, where $0 \leq c \leq n-1$, and $k \in \Z$. Since $y^{kn}$ is central in $G(m)$, for all $k \in \Z$, we can rewrite any geodesic such that all powers of $y^{n}$ are moved to the right hand side of the word.

\begin{defn}\label{defn:mod normal form}
     Any geodesic $g \in X^{\ast}$ can be written uniquely in the form $g =_{G} \left(hy^{c}, y^{kn}\right)$, where $hy^{c}$ is a geodesic representing an element of $F_{n} \rtimes \Z/n\Z$. We say $\left(hy^{c}, y^{kn}\right)$ is the \emph{modular normal form} for $g$. Note $w = hy^{c}y^{kn} \in X^{\ast}$ may not necessarily be geodesic in $G(m)$.
\end{defn}
For the remainder of this section we will use our modular normal form when representing group elements. We denote $G_{\mathrm{mod}}$ to be the set of all words in modular normal form representing group elements from $G(m) \cong (F_{n} \rtimes \Z/n\Z) \times \Z$. For any word $(hy^{c}, y^{kn}) \in G_{\mathrm{mod}}$ which represents an element $w \in G(m)$, we say $h$ is the \emph{free component} of $w$ (i.e. $h \in F_{n})$, $hy^{c}$ is the \emph{quotient element} of $w$, and $y^{kn}$ is the \emph{Garside power} of $w$.  

We now consider what happens when we twisted conjugate a word $u \in G_{\mathrm{mod}}$ by an element $w \in \Z$. 

\begin{prop}\label{prop:y shift formula}
    Let $u = (u_{1}y^{\alpha_{1}}, y^{\alpha_{2}}) \in G_{\mathrm{mod}}$ be a modular normal form, where $\alpha = \alpha_{1}+\alpha_{2}$. Let $\lambda \in \Z$. Then
    \[ u = u_{1}y^{\alpha} \sim_{\phi} \Phi_{-\e_{y}\lambda}(u_{1})y^{\alpha+\lambda(1-\e_{y})}.     \]
\end{prop}
\begin{proof}
    Let $u_{1} = x^{p_{1}}_{i_{1}}\dots x^{p_{q}}_{i_{q}} \in F_{n}$, and let $w = y^{\lambda}$. Then 
    \begin{align*}
        \phi(w)^{-1}uw &= y^{-\e_{y}\lambda}\cdot x^{p_{1}}_{i_{1}}\dots x^{p_{q}}_{i_{q}}\cdot y^{\alpha+\lambda} \\
        &=_{G} \Phi_{-\e_{y}\lambda}\left(x^{p_{1}}_{i_{1}}\dots x^{p_{q}}_{i_{q}}\right) y^{\alpha+\lambda(1-\e_{y})} \\
        &= \Phi_{-\e_{y}\lambda}(u_{1})y^{\alpha+\lambda(1-\e_{y})}.\tag*{\qedhere}
    \end{align*}
\end{proof}
Note that since $\Phi$ is an automorphism and $u_{1}$ is geodesic, then $\Phi_{-\e_{y}\lambda}(u_{1})$ is geodesic. In particular, when we twisted conjugate by an element $w \in \Z$, the length of the free component does not change. This allows us to define our first type of shift.
\begin{defn}\label{defn: len p y shift}
    Let $u = (u_{1}y^{\alpha_{1}}, y^{\alpha_{2}}), v = (v_{1}y^{\beta_{1}}, y^{\beta_{2}}) \in G_{\mathrm{mod}}$, where $\alpha = \alpha_{1}+\alpha_{2}$ and $\beta = \beta_{1}+\beta_{2}$. We say $u$ and $v$ are related by a \emph{$y$-shift} if $v_{1} = \Phi_{-\e_{y}\lambda}(u_{1})$ and $\beta = \alpha +\lambda(1-\e_{y})$, for some $\lambda \in \Z$. We denote a $y$-shift by $u \xleftrightarrow{y} v$, and note that if such a $y$-shift exists between $u$ and $v$, then $u \sim_{\phi} v$ by \cref{prop:y shift formula}. 
\end{defn}
One immediate consequence of \cref{prop:y shift formula} is the following result, due to the fact that any powers of $y^{n}$ are central in $G(m)$. 
\begin{cor}\label{cor: len p garside}
    Let $u = (u_{1}y^{\alpha_{1}}, y^{\alpha_{2}})  \in G_{\mathrm{mod}}$, where $\alpha = \alpha_{1}+\alpha_{2}$. Let $\lambda \in \Z$. Then $u = u_{1}y^{\alpha} \sim_{\phi} u_{1}y^{\alpha+\lambda n(1-\e_{y})}$.
\end{cor}
\comm{
\begin{rmk}\label{rmk:len p  y shift garside}
    One observation of \cref{cor: len p garside} is that any $y$-shifts by powers of $y^{p}$ do not change the free component. Indeed any powers of $y^{p}$ are central. 
\end{rmk}}
We now consider twisted conjugacy by an arbitrary element $w \in G_{\mathrm{mod}}$. Let \\$u = \left(x^{p_{1}}_{i_{1}}\dots x^{p_{q}}_{i_{q}}y^{\alpha_{1}}, y^{\alpha_{2}}\right) \in G_{\mathrm{mod}}$ be a modular normal form, and suppose we twisted conjugate $u$ by an element $w = \left(x^{t_{1}}_{k_{1}}\dots x^{t_{z}}_{k_{z}}y^{\lambda_{1}}, y^{\lambda_{2}}\right) \in G_{\mathrm{mod}}$, where $z \geq 1$. Let $\sigma = \sum_{i=1}^{z} t_{i}$. We first compute $\phi(w)$, which gives us
\begingroup
\addtolength{\jot}{0.5em}
\begin{align*}
        \phi(w) &=_{G} \phi_{F}\left(x^{t_{1}}_{k_{1}}\right)y^{t_{1}d}\phi_{F}\left(x^{t_{2}}_{k_{2}}\right)y^{t_{2}d}\dots \phi_{F}\left(x^{t_{z}}_{k_{z}}\right)y^{t_{z}d}y^{\e_{y}(\lambda_{1}+\lambda_{2})} \\
        &=_{G} \phi_{F}\left(x^{t_{1}}_{k_{1}}\right)\Phi_{t_{1}d}\left(\phi_{F}\left(x^{t_{2}}_{k_{2}}\right)\right)\Phi_{(t_{1}+t_{2})d}\left(\phi_{F}\left(x^{t_{3}}_{k_{3}}\right)\right)\dots \Phi_{(t_{1} + \dots + t_{z-1})d}\left(\phi_{F}\left(x^{t_{z}}_{k_{z}}\right)\right)\cdot y^{\sigma d + \e_{y}(\lambda_{1}+
        \lambda_{2})},
    \end{align*}
    \endgroup
We then have
\begin{align*}
        \phi(w)^{-1} &=_{G} y^{-(\sigma d + \e_{y}(\lambda_{1}+\lambda_{2}))}\left(\Phi_{(t_{1}+t_{2} + \dots + t_{z-1})d}\left(\phi_{F}\left(x^{t_{z}}_{k_{z}}\right)\right)\right)^{-1}\dots \left(\phi_{F}\left(x^{t_{1}}_{k_{1}}\right)\right)^{-1} \\
        &=_{G} \Phi_{-(t_{z}d+\e_{y}\lambda_{1})}\left(\phi_{F}\left(x^{t_{z}}_{k_{z}}\right)^{-1}\right)\dots \Phi_{-(\sigma d + \e_{y}\lambda_{1})}\left(\phi_{F}\left(x^{t_{1}}_{k_{1}}\right)^{-1}\right)\cdot y^{-(\sigma d + \e_{y}(\lambda_{1}+\lambda_{2}))},
    \end{align*}
    noting that $y^{-\e_{y}\lambda_{2}}$ is central. Therefore
    \begin{align*}
        \phi(w)^{-1}uw &=_{G} \Phi_{-(t_{z}d + \e_{y}\lambda_{1})}\left(\phi_{F}\left(x^{t_{z}}_{k_{z}}\right)^{-1}\right)\dots \Phi_{-(\sigma d + \e_{y}\lambda_{1})}\left(\phi_{F}\left(x^{t_{1}}_{k_{1}}\right)^{-1}\right) y^{-(\sigma d + \e_{y}(\lambda_{1}+\lambda_{2}))}\\
        & \cdot x^{p_{1}}_{i_{1}}x^{p_{2}}_{i_{2}}\dots x^{p_{q}}_{i_{q}}y^{\alpha_{1}}y^{\alpha_{2}}\cdot x^{t_{1}}_{k_{1}}x^{t_{2}}_{k_{2}}\dots x^{t_{z}}_{k_{z}}y^{\lambda_{1}}y^{\lambda_{2}}.
    \end{align*}
For notation, let $\gamma = -(\sigma d + \e_{y}\lambda_{1})$ and $\delta = -(\gamma + \alpha_{1} + \lambda_{1})$. We first move all $y$ powers to the right hand side, which gives us
\begin{equation}\label{eqn:non len p general form}
\begin{split}
    \phi(w)^{-1}uw &=_{G} \Phi_{-(t_{z}d + \e_{y}\lambda_{1})}\left(\phi_{F}\left(x^{t_{z}}_{k_{z}}\right)^{-1}\right)\dots \Phi_{\gamma}\left(\phi_{F}\left(x^{t_{1}}_{k_{1}}\right)^{-1}\right) \\
    &\cdot \Phi_{\gamma}\left(x^{p_{1}}_{i_{1}}\right)\dots \Phi_{\gamma}\left(x^{p_{q}}_{i_{q}}\right)\Phi_{\gamma+\alpha_{1}}\left(x^{t_{1}}_{k_{1}}\right)\dots \Phi_{\gamma+\alpha_{1}}\left(x^{t_{z}}_{k_{z}}\right)y^{-\delta}y^{\alpha_{2}+\lambda_{2}(1-\e_{y})}.
\end{split}
    \end{equation}
We first consider the $y$ exponent of $\phi(w)^{-1}uw$. Since $-\delta = \alpha_{1}+\lambda_{1}(1-\e_{y}) - \sigma d$, we can determine the value of $-\delta$ in the subcase where $d = 0 \Mod{n}$. Indeed if this is the case, we can rewrite $\phi(w)^{-1}uw =_{G} (v_{1}, v_{2}) \in G_{\mathrm{mod}}$ in our modular normal form, where our $y$ exponent of $v_{1}$ equals $\alpha_{1}+\lambda_{1}(1-\e_{y}) \Mod{n}$. This allows us to immediately rule out some cases of twisted conjugate elements when $d=0 \Mod{n}$.  

\begin{cor}\label{cor:len p compare y}
    Let $u = \left(x^{p_{1}}_{i_{1}}\dots x^{p_{q}}_{i_{q}}y^{\alpha_{1}}, y^{\alpha_{2}}\right), v = \left(x^{r_{1}}_{j_{1}}\dots x^{r_{s}}_{j_{s}}y^{\beta_{1}}, y^{\beta_{2}}\right) \in G_{\mathrm{mod}}$ be modular normal forms. Let $\phi \in \mathrm{Out}(G(m))$ be defined as in \cref{even:all auto forms}, where $d=0 \Mod{n}$. 
    \begin{enumerate}
        \item[(i)] Suppose $\e_{y} = 1$. If $\alpha_{1} \neq \beta_{1}$ or $\alpha_{2} \neq \beta_{2}$, then $u \not \sim_{\phi} v$.
        \item[(ii)] Suppose $\e_{y} = -1$. If $\beta_{1} \neq \alpha_{1} + 2\lambda \Mod{n}$ for all $\lambda \in \Z /n\Z$, then $u \not \sim_{\phi} v$.
    \end{enumerate}
\end{cor}
Returning to the more general case, we set up some notation for working with \cref{eqn:non len p general form}. We let
\begin{equation}\label{eqn:a1 a2 a3}
\begin{split}
    a_{1} &= \Phi_{-(t_{z}d + \e_{y}\lambda_{1})}\left(\phi_{F}\left(x^{t_{z}}_{k_{z}}\right)^{-1}\right)\dots \Phi_{\gamma}\left(\phi_{F}\left(x^{t_{1}}_{k_{1}}\right)^{-1}\right), \\
    a_{2} &= \Phi_{\gamma}\left(x^{p_{1}}_{i_{1}}\dots x^{p_{q}}_{i_{q}}\right), \\
    a_{3} &= \Phi_{\gamma+\alpha_{1}}\left(x^{t_{1}}_{k_{1}} \dots x^{t_{z}}_{k_{z}}\right).
\end{split}
\end{equation}    
Here $a_{2}$ and $a_{3}$ are geodesic since $\Phi$ is a homomorphism, however it is not clear whether $a_{1}$ is also geodesic. We prove this is indeed true. 
\begin{lemma}\label{lem:preserve reduction}
    Let $w = x^{t_{1}}_{k_{1}}\dots x^{t_{z}}_{k_{z}} \in X_{n}^{\ast}$, where $t_{i} \in \{\pm 1\}, k_{i} \in \{0,\dots, n-1\}$ $(1 \leq i \leq z)$. Then $w$ is geodesic if and only if the word
    \[ w' = \phi_{F}\left(x^{t_{1}}_{k_{1}}\right)\Phi_{t_{1}d}\left(\phi_{F}\left(x^{t_{2}}_{k_{2}}\right)\right)\dots \Phi_{(t_{1}+\dots + t_{z-1})d}\left(\phi_{F}\left(x^{t_{z}}_{k_{z}}\right)\right) 
    \]
    is geodesic. 
\end{lemma}

\begin{proof}
    For the forward direction, suppose there exists $1 \leq i \leq z-1$ such that
    \[ \Phi_{(t_{1}+\dots + t_{(i-1)})d}\left(\phi_{F}\left(x^{t_{i}}_{k_{i}}\right)\right)^{-1} = \Phi_{(t_{1}+\dots + t_{i})d}\left(\phi_{F}\left(x^{t_{(i+1)}}_{k_{(i+1)}}\right)\right).
    \]
    Since $\Phi$ is a bijection, this reduces to $\phi_{F}\left(x^{t_{i}}_{k_{i}}\right)^{-1} = \Phi_{t_{i}d}\left(\phi_{F}\left(x^{t_{(i+1)}}_{k_{(i+1)}}\right)\right)$. If $t_{i} = 1$, then $t_{(i+1)} = -1$, and our relation becomes 
    \[ x^{-\e_{x}}_{\left[\e_{y}k_{i}\right]} = x^{-\e_{x}}_{[\e_{y}k_{(i+1)}+d-d]} \; \Leftrightarrow \; k_{i} = k_{(i+1)}.
    \]
    If this holds, then the subword $x^{t_{i}}_{k_{i}}x^{t_{(i+1)}}_{k_{(i+1)}} =_{G} 1$, and so $w$ is not geodesic, which gives a contradiction. Similarly suppose $t_{i} = -1$, and $t_{(i+1)} = 1$. We have 
    \[ x^{\e_{x}}_{[\e_{y}k_{i}+d]} = x^{\e_{x}}_{[\e_{y}k_{(i+1)}+d]} \; \Leftrightarrow \; k_{i} = k_{(i+1)},
    \]
    which again leads to a contradiction. The reverse direction follows a symmetric proof.
\end{proof}
\begin{cor}
    Let $a_{1} = \Phi_{-(t_{z}d + \e_{y}\lambda_{1})}\left(\phi_{F}\left(x^{t_{z}}_{k_{z}}\right)^{-1}\right)\dots \Phi_{\gamma}\left(\phi_{F}\left(x^{t_{1}}_{k_{1}}\right)^{-1}\right),$ as defined in \cref{eqn:a1 a2 a3}. Then $a_{1}$ is geodesic.
\end{cor}

\begin{proof}
    Since $a_{1} = \Phi_{\gamma}\left(\left(w'\right)^{-1}\right)$ from \cref{lem:preserve reduction}, then $a_{1}$ is geodesic.
\end{proof}
We now introduce a notion of twisted cyclic reduction, similar to that of \cite{crowe_twisted_2024}. For an element in modular normal form, we consider when the length of the free component is minimal with respect to its twisted conjugacy class. Recall for any word $u = (u_{1}y^{\alpha_{1}}, y^{\alpha_{2}}) \in G_{\mathrm{mod}}$, $u_{1} \in F_{n}$ is the free component of $u$. 
\begin{defn}\label{defn: twisted CR}
    Let $u = (u_{1}y^{\alpha_{1}}, y^{\alpha_{2}}) \in G_{\mathrm{mod}}$ be a modular normal form. Suppose \\$\phi(w)^{-1}uw =_{G} (v_{1}y^{\beta_{1}}, y^{\beta_{2}}) \in G_{\mathrm{mod}}$ for some $w \in G_{\mathrm{mod}}$. We say $u$ is \emph{$\phi$-cyclically reduced} ($\phi$-CR) if for all $w \in G_{\mathrm{mod}}$, $l(v_{1}) \geq l(u_{1})$.  
\end{defn}
\begin{prop}\label{prop:nonlen p cyc reduced}
    Let $u = (u_{1}y^{\alpha_{1}}, y^{\alpha_{2}}) \in G_{\mathrm{mod}}$ be a modular normal form, where \\$u_{1} = x^{p_{1}}_{i_{1}}\dots x^{p_{q}}_{i_{q}} \in F_{n}$. Then $u$ is $\phi$-CR if and only if $u_{1}$ is not of the form
    \[ u_{1} = 
    \begin{cases}
        x^{p_{1}}_{i_{1}}\dots x^{-\e_{x}p_{1}}_{[\e_{y}i_{1}-\alpha_{1}]} & \e_{x}p_{1} = 1, \\
        x^{p_{1}}_{i_{1}}\dots x^{-\e_{x}p_{1}}_{\left[\e_{y}(i_{1}-d) - \alpha_{1}\right]} & \e_{x}p_{1} = -1.
    \end{cases} \]
\end{prop}

\begin{proof}
    We start with the contrapositive of the reverse direction. Suppose we twisted conjugate by an element $w = \left(x^{t_{1}}_{k_{1}}\dots x^{t_{z}}_{k_{z}}y^{\lambda_{1}}, y^{\lambda_{2}}\right) \in G_{\mathrm{mod}}$. From \cref{eqn:non len p general form}, we have
\begin{align*}
    \phi(w)^{-1}uw &=_{G} \Phi_{-(t_{z}d + \e_{y}\lambda_{1})}\left(\phi_{F}\left(x^{t_{z}}_{k_{z}}\right)^{-1}\right)\dots \Phi_{\gamma}\left(\phi_{F}\left(x^{t_{1}}_{k_{1}}\right)^{-1}\right) \\
    &\cdot \Phi_{\gamma}\left(x^{p_{1}}_{i_{1}}\right)\dots \Phi_{\gamma}\left(x^{p_{q}}_{i_{q}}\right)\Phi_{\gamma+\alpha_{1}}\left(x^{t_{1}}_{k_{1}}\right)\dots \Phi_{\gamma+\alpha_{1}}\left(x^{t_{z}}_{k_{z}}\right)y^{-\delta}y^{\alpha_{2}+\lambda_{2}(1-\e_{y})}.
\end{align*}
Let $\phi(w)^{-1}uw =_{G} (v_{1}y^{\beta_{1}}, y^{\beta_{2}}) \in G_{\mathrm{mod}}$. In particular, let $v_{1} =_{G} a_{1}a_{2}a_{3}$, where $a_{1}, a_{2}, a_{3} \in X_{n}^{\ast}$ are defined in \cref{eqn:a1 a2 a3}. Since $u$ is not $\phi$-CR, then $a_{1}a_{2}a_{3}$ must freely reduce to an element of shorter length than $u_{1}$. Recall $a_{1}, a_{2}$ and $a_{3}$ are all geodesic, since $\Phi$ is a homomorphism and by \cref{lem:preserve reduction}. Therefore there exists both of the following cancellations:
\[ \Phi_{\gamma}\left(\phi_{F}\left(x^{t_{1}}_{k_{1}}\right)\right) = \Phi_{\gamma}\left(x^{p_{1}}_{i_{1}}\right), \quad \text{and} \quad \Phi_{\gamma}\left(x^{-p_{q}}_{i_{q}}\right) = \Phi_{\gamma+\alpha_{1}}\left(x^{t_{1}}_{k_{1}}\right).
\]
Since $\Phi$ is a homomorphism, we can reduce these two relations to
\[ \phi_{F}\left(x^{t_{1}}_{k_{1}}\right) = x^{p_{1}}_{i_{1}}, \quad \text{and} \quad x^{-p_{q}}_{i_{q}} = \Phi_{\alpha_{1}}\left(x^{t_{1}}_{k_{1}}\right).
\]
When considering powers, we have $p_{1} = \e_{x}t_{1} = -\e_{x}p_{q}$, and so $p_{q} = -\e_{x}p_{1}$. If $t_{1} = 1$, then $\phi_{F}\left(x^{t_{1}}_{k_{1}}\right) = x^{\e_{x}}_{[\e_{y}k_{1}]} = x^{\e_{x}}_{i_{1}}$, and so $i_{1} = \e_{y}k_{1} \Mod{n}$. Also, $i_{q} = k_{1}-\alpha_{1} = \e_{y}i_{1}-\alpha_{1} \Mod{n}$ (recall $\e^{2}_{y} = 1$). This implies that $u_{1} = x^{p_{1}}_{i_{1}}\dots x^{-\e_{x}p_{1}}_{[\e_{y}i_{1}-\alpha_{1}]}.$ Otherwise, if $t_{1} = -1$, then $\phi_{F}\left(x^{t_{1}}_{k_{1}}\right) = x^{-\e_{x}}_{[\e_{y}k_{1}+d]} = x^{-\e_{x}}_{i_{1}}$, and so $i_{1} = \e_{y}k_{1}+d \Mod{n}$. Also, $i_{q} = k_{1}-\alpha_{1} = \e_{y}(i_{1}-d) - \alpha_{1} \Mod{n}$. This implies that $u_{1} = x^{p_{1}}_{i_{1}}\dots x^{-\e_{x}p_{1}}_{[\e_{y}(i_{1}-d) - \alpha_{1}]}.$

For the forward direction, suppose $u_{1} = x^{p_{1}}_{i_{1}}\dots x^{-\e_{x}p_{1}}_{[\e_{y}i_{1}-\alpha_{1}]}$, where $\e_{x}p_{1} = 1$. We can twisted conjugate by $w = x^{\e_{x}p_{1}}_{[\e_{y}i_{1}]}$ which gives us 
\begingroup
\addtolength{\jot}{0.5em}
\begin{align*}
    \phi(w)^{-1}uw &= \phi\left(x^{\e_{x}p_{1}}_{[\e_{y}i_{1}]}\right)^{-1}x^{p_{1}}_{i_{1}}\dots x^{-\e_{x}p_{1}}_{[\e_{y}i_{1}-\alpha_{1}]}y^{\alpha_{1}}y^{\alpha_{2}}x^{\e_{x}p_{1}}_{[\e_{y}i_{1}]} \\
    &=_{G} y^{-d}x^{-p_{1}}_{i_{1}}x^{p_{1}}_{i_{1}}\dots x^{-\e_{x}p_{1}}_{[\e_{y}i_{1}-\alpha_{1}]}x^{\e_{x}p_{1}}_{[\e_{y}i_{1}-\alpha_{1}]}y^{\alpha_{1}}y^{\alpha_{2}}\\
    &=_{G} y^{-d}x^{p_{2}}_{i_{2}}\dots x^{p_{q-1}}_{i_{q-1}}y^{\alpha_{1}+\alpha_{2}} \\
    &=_{G} \Phi_{-d}\left(x^{p_{2}}_{i_{2}}\dots x^{p_{q-1}}_{i_{q-1}}\right)y^{\alpha_{1}+\alpha_{2}-d},
\end{align*}
\endgroup
which has a shorter free component as required. A similar proof holds in the case where $u_{1} = x^{p_{1}}_{i_{1}}\dots x^{-\e_{x}p_{1}}_{[\e_{y}(i_{1}-d) - \alpha_{1}]}$, where $\e_{x}p_{1} = -1$. 
\end{proof}
We now define an equivalent notion of \cref{defn: len p y shift}, in terms of moving letters from the free component only. Extra care must be taken to track changes to the $y$ power of an element, to ensure we remain within the same twisted conjugacy class.

Suppose we have a word $u = (u_{1}y^{\alpha_{1}}, y^{\alpha_{2}}) \in G_{\mathrm{mod}}$ which is $\phi$-CR. Suppose $u_{1} = u_{11}u_{12} \in F_{n}$, where $u_{11}, u_{12} \in F_{n}$ are (possible empty) subwords of the free component of $u$. First, suppose we shift the subword $u_{12}$ from the back to the front of the free component. To do this, we apply the relation $\phi\left(\Phi_{-\alpha_{1}}\left(u_{12}\right)\right)\cdot u \cdot \Phi_{-\alpha_{1}}\left(u^{-1}_{12}\right)$, in order to preserve the twisted conjugacy class. This gives us
\begingroup
\addtolength{\jot}{0.5em}
\begin{align*}
    \phi(\Phi_{-\alpha_{1}}(u_{12}))\cdot u \cdot \Phi_{-\alpha_{1}}\left(u^{-1}_{12}\right) &= \phi(\Phi_{-\alpha_{1}}(u_{12}))u_{11}u_{12}y^{\alpha_{1}}y^{\alpha_{2}}\Phi_{-\alpha_{1}}\left(u^{-1}_{12}\right) \\
    &=_{G} \phi(\Phi_{-\alpha_{1}}(u_{12}))u_{11}u_{12}u^{-1}_{12}y^{\alpha_{1}}y^{\alpha_{2}} \\
    &=_{G} \phi(\Phi_{-\alpha_{1}}(u_{12}))u_{11}y^{\alpha_{1}}y^{\alpha_{2}}.
\end{align*}
\endgroup
Now suppose we shift the subword $u_{11}$ from the front to the back of the free component. Similarly we apply the relation $u^{-1}_{11}\cdot u \cdot \phi^{-1}(u_{11}),$ in order to preserve the twisted conjugacy class. This leaves us with
\[ u^{-1}_{11}\cdot u \cdot \phi^{-1}(u_{11}) = u^{-1}_{11}u_{11}u_{12}y^{\alpha_{1}}y^{\alpha_{2}}\phi^{-1}(u_{11}) =_{G} u_{12}y^{\alpha_{1}}\phi^{-1}(u_{11})y^{\alpha_{2}}.
\]

\begin{defn}\label{defn:x shift}
    Let $u = (u_{11}u_{12}y^{\alpha_{1}}, y^{\alpha_{2}}) \in G_{\mathrm{mod}}$ be $\phi$-CR. We say an element $v \in G_{\mathrm{mod}}$ is related to $u$ by an \emph{$x$-shift} if $v$ is equivalent to either
    \[ v =_{G} \phi\left(\Phi_{-\alpha_{1}}\left(u_{12}\right)\right)u_{11}y^{\alpha_{1}}y^{\alpha_{2}}, \quad \text{or} \quad v =_{G} u_{12}y^{\alpha_{1}}\phi^{-1}\left(u_{11}\right)y^{\alpha_{2}}.
    \]
    We denote an $x$-shift by $u \xmapsto{x} v$, and note that if such a $x$-shift exists between $u$ and $v$, then $u \sim_{\phi} v$ by the discussion above.  
\end{defn}
\begin{rmk}
    It is worth emphasising that $v$ is \emph{equivalent} to these elements. In particular, after applying an $x$-shift, the $y$ exponent may have changed after moving $y$ terms from left to right. In general, $x$-shifts cannot be reversed, unlike $y$-shifts, as defined in \cref{defn: len p y shift}, which can be reversed.
\end{rmk}
We provide an example to demonstrate how $x$-shifts and $y$-shifts work in practice. 

\begin{exmp}\label{exmp:non len p algorithm}
    We recall \cref{exmp:first}, and let $n=3$, i.e. $F_{n} = \langle x_{0}, x_{1}, x_{2} \rangle$. Let $\phi \colon x_{0} \mapsto x_{0}y^{4}, \; y \mapsto y$, i.e. $\e_{x} = \e_{y} =1$ and $d = 4$. We consider the geodesic $u = x_{0}x^{-1}_{2}y^{2} =_{G} \left(x_{0}x^{-1}_{2}y^{2}, 1_{G}\right) \in G_{\mathrm{mod}}$. We note that since $\e_{y} = 1$, no $y$-shift will change the $y$ exponent of $u$. For example, we could apply the following $y$-shift to $u$:
    \[ u =_{G} \left(x_{0}x^{-1}_{2}y^{2}, 1_{G}\right) \xleftrightarrow{y} \left(\Phi_{-1}\left(x_{0}x^{-1}_{2}\right)y^{2}, 1_{G}\right) = \left(x_{1}x^{-1}_{0}y^{2}, 1_{G}\right).
    \]
    Alternatively, we could apply an $x$-shift to $u$ as follows, by shifting the subword $x^{-1}_{2}$ from the back to the front of the free component.
    \begin{align*}
        u =_{G} \left(x_{0}x^{-1}_{2}y^{2}, 1_{G}\right) &\xmapsto{x} \left(\phi\left(\Phi_{-2}\left(x^{-1}_{2}\right)\right)x_{0}y^{2}, 1_{G}\right) \\
        &= \left(\phi\left(x^{-1}_{1}\right)x_{0}y^{2}, 1_{G}\right) \\
        &=_{G} \left(x^{-1}_{2}y^{-4}x_{0}y^{2}, 1_{G}\right) \\
        &=_{G} \left(x^{-1}_{2}y^{-1}x_{0}y^{2}, y^{-3}\right)\\
        &=_{G} \left(x^{-1}_{2}x_{1}y, y^{-3}\right).
    \end{align*}
\end{exmp}
With these definitions of $x$-shifts and $y$-shifts, we can show that a finite sequence of shifts exists between any two $\phi$-CR elements which are twisted conjugate. 

\begin{thm}\label{thm:non len p main result}
    Let $u = (u_{1}y^{\alpha_{1}}, y^{\alpha_{2}}), v = (v_{1}y^{\beta_{1}}, y^{\beta_{2}}) \in G_{\mathrm{mod}}$ be modular normal forms which are $\phi$-CR. Then $u \sim_{\phi} v$ if and only if $l\left(u_{1}\right) = l\left(v_{1}\right)$, and $u$ and $v$ are related by a finite sequence of $x$-shifts and $y$-shifts.
\end{thm}
The proof strategy is similar to that of \cite[Theorem 3.13]{crowe_conjugacy_2023} and \cite{bardakov_conjugacy_2005}. We require one further algebraic result before proving this theorem. 
\begin{lemma}\label{lemma:claims for main proof}
    Let $a_{1}, a_{3} \in F_{n}$ be defined from \cref{eqn:a1 a2 a3}, i.e.
    \[ a_{1} = \Phi_{-(t_{z}d + \e_{y}\lambda_{1})}\left(\phi_{F}\left(x^{t_{z}}_{k_{z}}\right)^{-1}\right)\dots \Phi_{\gamma}\left(\phi_{F}\left(x^{t_{1}}_{k_{1}}\right)^{-1}\right), \quad
    a_{3} = \Phi_{\gamma+\alpha_{1}}\left(x^{t_{1}}_{k_{1}} \dots x^{t_{z}}_{k_{z}}\right).
    \]
    Then $\phi(\Phi_{\delta}(a_{3})) =_{G} \Phi_{\sigma d}\left(a^{-1}_{1}\right)y^{\sigma d}$, where $\sigma = \sum_{i=1}^{z} t_{i}$. Moreover, $a_{1} =_{G} \Phi_{\e_{y}\delta}\left(\left[\phi\left(a^{-1}_{3}\right)\right]_{F}\right)$.
\end{lemma}

\begin{proof}
    Recall $\gamma = -(\sigma d + \e_{y}\lambda_{1})$ and $\delta = -(\gamma + \alpha_{1} + \lambda_{1})$, and so $\delta + \gamma + \alpha_{1} = -\lambda_{1}$. Then
    \begingroup
\addtolength{\jot}{0.5em}
\begin{align*}
    \phi(\Phi_{\delta}(a_{3})) &= \phi\left(\Phi_{-\lambda_{1}}\left(x^{t_{1}}_{k_{1}}\dots x^{t_{z}}_{k_{z}}\right)\right) \\
    &= \phi\left(x^{t_{1}}_{[k_{1}+\lambda_{1}]}\dots x^{t_{z}}_{[k_{z}+\lambda_{1}]}\right) \\
    &=_{G} \phi_{F}\left(x^{t_{1}}_{[k_{1}+\lambda_{1}]}\right)\Phi_{t_{1}d}\left(\phi_{F}\left(x^{t_{2}}_{[k_{2}+\lambda_{1}]}\right)\right)\dots \Phi_{(t_{1}+\dots + t_{z -1})d}\left(\phi_{F}\left(x^{t_{z}}_{[k_{z}+\lambda_{1}]}\right)\right)y^{\sigma d} \\
    &= \Phi_{\sigma d}\left(\Phi_{-\sigma d}\left(\phi_{F}\left(x^{t_{1}}_{[k_{1}+\lambda_{1}]}\right)\right)\dots \Phi_{-t_{z}d}\left(\phi_{F}\left(x^{t_{z}}_{[k_{z}+\lambda_{1}]}\right)\right)\right)y^{\sigma d}.
\end{align*}
\endgroup
By $(i)$ of \cref{prop:non len p 1}, we have $\phi_{F}\left(x^{t_{i}}_{[k_{i}+\lambda_{1}]}\right) = \phi_{F}\left(\Phi_{-\lambda_{1}}\left(x^{t_{i}}_{k_{i}}\right)\right) = \Phi_{-\e_{y}\lambda_{1}}\left(\phi_{F}\left(x^{t_{i}}_{k_{i}}\right)\right)$, for all $0 \leq i \leq z$. Hence
\[  \phi(\Phi_{\delta}(a_{3})) =_{G} \Phi_{\sigma d}\left(\Phi_{-(\sigma d + \e_{y}\lambda_{1})}\left(\phi_{F}\left(x^{t_{1}}_{k_{1}}\right)\right)\dots \Phi_{-(t_{z}d + \e_{y}\lambda_{1})}\left(\phi_{F}\left(x^{t_{z}}_{k_{z}}\right)\right)\right)y^{\sigma d} = \Phi_{\sigma d}\left(a^{-1}_{1}\right)y^{\sigma d}.
\]
This implies that 
$[\phi(\Phi_{\delta}(a_{3}))]_{F} =_{G} \Phi_{\sigma d}\left(a^{-1}_{1}\right)$. By $(iv)$ of \cref{prop:non-len p results 2}, we have
\begingroup
\addtolength{\jot}{0.5em}
\begin{align*}
    \Phi_{\e_{y}\delta}\left(\left[\phi(a_{3})\right]_{F}\right) &=_{G} \Phi_{\sigma d}\left(a^{-1}_{1}\right) \\
    \Rightarrow \Phi_{\e_{y}\delta - \sigma d}\left(\left[\phi(a_{3})\right]_{F}\right) &=_{G} a^{-1}_{1} \\
    \Rightarrow \Phi_{\e_{y}\delta - \sigma d}\left(\left[\phi(a_{3})\right]_{F}\right)^{-1} &=_{G} a_{1}. 
\end{align*}
\endgroup
By $(v)$ of \cref{prop:non-len p results 2}, we have $a_{1} =_{G} \Phi_{\e_{y}\delta - \sigma d + \sigma d}\left(\left[\phi(a_{3})^{-1}\right]_{F}\right) =  \Phi_{\e_{y}\delta}\left(\left[\phi\left(a^{-1}_{3}\right)\right]_{F}\right),$ which completes the proof.
\end{proof}

\begin{proof}[Proof of \cref{thm:non len p main result}]
    The reverse direction is immediate from \cref{defn: len p y shift} and \cref{defn:x shift}, and the requirement for $l\left(u_{1}\right) = l\left(v_{1}\right)$ follows from the $\phi$-CR assumption. For the forward direction, suppose $\phi(w)^{-1}uw =_{G} \left(v_{1}y^{\beta_{1}}, y^{\beta_{2}}\right)$, for some element \\$w = \left(x^{t_{1}}_{k_{1}}\dots x^{t_{z}}_{k_{z}}y^{\lambda_{1}}, y^{\lambda_{2}}\right) \in G_{\mathrm{mod}}$. For notation, let $u = (u_{1}y^{\alpha_{1}}, y^{\alpha_{2}}) \in G_{\mathrm{mod}}$, where \\$u_{1} = x^{p_{1}}_{i_{1}}\dots x^{p_{q}}_{i_{q}} \in F_{n}$. Let $\alpha=\alpha_{1}+\alpha_{2}$ and $\lambda=\lambda_{1}+\lambda_{2}$. From \cref{eqn:non len p general form}, we can assume that 
    \[ v_{1} =_{G} \Phi_{-(t_{z}d + \e_{y}\lambda_{1})}\left(\phi_{F}\left(x^{t_{z}}_{k_{z}}\right)^{-1}\right)\dots \Phi_{\gamma}\left(\phi_{F}\left(x^{t_{1}}_{k_{1}}\right)^{-1}\right) \Phi_{\gamma}\left(x^{p_{1}}_{i_{1}}\dots x^{p_{q}}_{i_{q}}\right)\Phi_{\gamma+\alpha_{1}}\left(x^{t_{1}}_{k_{1}}\dots x^{t_{z}}_{k_{z}}\right),
    \]
    where $\gamma = -(\sigma d + \e_{y}\lambda_{1})$. Recall from \cref{eqn:a1 a2 a3} that
    \begin{align*}
    a_{1} &= \Phi_{-(t_{z}d + \e_{y}\lambda_{1})}\left(\phi_{F}\left(x^{t_{z}}_{k_{z}}\right)^{-1}\right)\dots \Phi_{\gamma}\left(\phi_{F}\left(x^{t_{1}}_{k_{1}}\right)^{-1}\right), \\
    a_{2} &= \Phi_{\gamma}\left(x^{p_{1}}_{i_{1}}\dots x^{p_{q}}_{i_{q}}\right), \\
    a_{3} &= \Phi_{\gamma+\alpha_{1}}\left(x^{t_{1}}_{k_{1}} \dots x^{t_{z}}_{k_{z}}\right),
\end{align*}
and so $v_{1} =_{G} a_{1}a_{2}a_{3}$, where $a_{1}, a_{2}, a_{3} \in F_{n}$ are geodesic. Since $l\left(u_{1}\right) = l\left(v_{1}\right)$, there exists precisely $z$ free cancellations in $a_{1}a_{2}a_{3}$. Moreover, free cancellation occurs in $a_{1}a_{2}$ or $a_{2}a_{3}$, but not both. We consider whether $a_{2}$ fully cancels in $a_{1}a_{2}a_{3}$.

\textbf{Case 1:} $a_{2}$ is not fully cancelling.\\
Suppose there are cancellations in $a_{1}a_{2}$ only. Since $u$ is $\phi$-CR, $a_{1}$ must fully cancel, and so we can write $a_{2} = a^{-1}_{1}a_{22}$. Then $v =_{G} a_{22}a_{3}y^{-\delta}y^{\alpha_{2}+\lambda_{2}(1-\e_{y})}$, where $\delta = -(\gamma + \alpha_{1} + \lambda_{1})$. We consider the element
\[ \overline{v} =_{G} \phi(\Phi_{\delta}(a_{3}))\cdot v \cdot \Phi_{\delta}(a_{3})^{-1}.
\]
This is precisely an $x$-shift of $v$, which gives us
\begin{align*}
    \phi(\Phi_{\delta}(a_{3}))\cdot v \cdot \Phi_{\delta}(a_{3})^{-1} &=_{G} \phi(\Phi_{\delta}(a_{3}))\cdot a_{22}a_{3}y^{-\delta}y^{\alpha_{2}+\lambda_{2}(1-\e_{y})} \cdot \Phi_{\delta}(a_{3})^{-1} \\
    &=_{G} \phi(\Phi_{\delta}(a_{3}))\cdot a_{22}a_{3}a^{-1}_{3}y^{-\delta}y^{\alpha_{2}+\lambda_{2}(1-\e_{y})} \\
    &=_{G}  \phi(\Phi_{\delta}(a_{3}))\cdot a_{22}y^{-\delta}y^{\alpha_{2}+\lambda_{2}(1-\e_{y})},
\end{align*}
and so the subword $a_{3}$ has been moved from the back to the front of the free component. Using \cref{lemma:claims for main proof}, we now have
\begin{align*}
    \overline{v} &=_{G}  \phi(\Phi_{\delta}(a_{3}))\cdot a_{22}y^{-\delta}y^{\alpha_{2}+\lambda_{2}(1-\e_{y})} \\
    &=_{G} \Phi_{\sigma d}\left(a^{-1}_{1}\right)y^{\sigma d}a_{22}y^{-\delta}y^{\alpha_{2}+\lambda_{2}(1-\e_{y})} \\
    &=_{G} \Phi_{\sigma d}\left(a^{-1}_{1}a_{22}\right)y^{\alpha+\lambda(1-\e_{y})} \\
    &= \Phi_{\sigma d}(a_{2})y^{\alpha+\lambda(1-\e_{y})} \\
    &= \Phi_{-\e_{y}\lambda_{1}}(u_{1})y^{\alpha+\lambda(1-\e_{y})}.
\end{align*}
Note $a_{2} = \Phi_{\gamma}(u_{1})$ and $\sigma d + \gamma = -\e_{y}\lambda_{1}$. Then $\overline{v} \xleftrightarrow{y} u$ (by twisted conjugating by $y^{-\lambda}$), and so $v \xmapsto{x} \overline{v} \xleftrightarrow{y} u$ as required. A similar method holds in the case where there are cancellations in $a_{2}a_{3}$ only.

\textbf{Case 2:} $a_{2}$ is fully cancelling. \\
First suppose $a_{2}$ fully cancels in $a_{3}$. We let $a_{3} = a^{-1}_{2}a_{31}$, and so $\phi(a_{31})^{-1}\phi(a_{2})$. By \cref{lemma:claims for main proof}, we have 
\[ v_{1} =_{G} a_{1}a_{31} =_{G} \Phi_{\e_{y}\delta}\left(\left[\phi\left(a^{-1}_{3}\right)\right]_{F}\right)a_{31} = \Phi_{\e_{y}\delta}\left(\left[\phi\left(a^{-1}_{31}a_{2}\right)\right]_{F}\right)a_{31}.
\]
Let $\sigma_{31}$ denote the exponent sum of $a_{31}$. By \cref{lemma:1}, we have\\ $v_{1} =_{G} \Phi_{\e_{y}\delta}\left(\left[\phi\left(a^{-1}_{31}\right)\right]_{F}\right)\Phi_{\e_{y}\delta - \sigma_{31}d}\left(\left[\phi(a_{2})\right]_{F}\right)a_{31}.$ We can assume there exists further cancellations with $a_{31}$, otherwise $v$ would not be $\phi$-CR. 

First suppose $a_{31}$ fully cancels, and so $\Phi_{\e_{y}\delta - \sigma_{31}d}\left(\left[\phi(a_{2})\right]_{F}\right) = a_{4}a^{-1}_{31}.$ This implies that \\ $[\phi(a_{2})]_{F} =  \Phi_{-\e_{y}\delta + \sigma_{31}d}\left(a_{4}a^{-1}_{31}\right)$. By \cref{prop:non-len p results 2} and \cref{lemma:3}, this implies that
\[ a_{2} = [\phi^{-1}\left(\Phi_{-\e_{y}\delta + \sigma_{31}d}\left(a_{4}a^{-1}_{31}\right)\right)]_{F} = \Phi_{-\delta + \e_{y}\sigma_{31}d}\left([\phi^{-1}\left(a_{4}a^{-1}_{31}\right)]_{F}\right).
\]
Returning to our original elements $u$ and $v$, we now have
\[ u =_{G} u_{1}y^{\alpha_{1}}y^{\alpha_{2}}, \quad v =_{G} \Phi_{\e_{y}\delta}\left([\phi\left(a^{-1}_{31}\right)]_{F}\right)a_{4}y^{-\delta}y^{\alpha_{2}+\lambda_{2}(1-\e_{y})}.
\]
We first apply a $y$-shift to $u$ to get $\overline{u} =_{G} \Phi_{\sigma d}(a_{2})y^{\alpha+\lambda(1-\e_{y})}.$ Next, we apply an $x$-shift to cancel out the $\Phi_{\sigma d}(a_{2})$ fully. Consider the element
\[\phi(\Phi_{\delta}(a_{2}))\cdot \overline{u} \cdot \Phi_{\delta}\left(a^{-1}_{2}\right) =_{G} \phi(\Phi_{\delta}(a_{2}))y^{\alpha+\lambda(1-\e_{y})}
\]
(note $\delta + \alpha_{1} + \lambda_{1}(1-\e_{y}) = \sigma d \Mod{n})$. Let $\sigma_{2}, \sigma_{4}$ denote the exponent sums of $a_{2}$ and $a_{4}$ respectively. Since $a_{2} = \Phi_{-\delta + \e_{y}\sigma_{31}d}\left([\phi^{-1}\left(a_{4}a^{-1}_{31}\right)]_{F}\right)$, then $\sigma_{2} = \e_{x}(\sigma_{4}-\sigma_{31})$ by \cref{prop:non-len p results 2}. We then have
\begingroup
\addtolength{\jot}{0.5em}
\begin{align*}
    \phi(\Phi_{\delta}(a_{2})) &= \phi\left(\Phi_{\delta}\left(\Phi_{-\delta + \e_{y}\sigma_{31}d}\left(\left[\phi^{-1}\left(a_{4}a^{-1}_{31}\right)\right]_{F}\right)\right)\right) \\
    &= \phi\left(\Phi_{\e_{y}\sigma_{31}d}\left(\left[\phi^{-1}\left(a_{4}a^{-1}_{31}\right)\right]_{F}\right)\right) \\
    &= \phi\left(\left[\phi^{-1}\left(\Phi_{\sigma_{31}d}\left(a_{4}a^{-1}_{31}\right)\right)\right]_{F}\right)  \\
    &=_{G} \Phi_{\sigma_{31}d}\left(a_{4}a^{-1}_{31}\right)y^{\e_{x}d(\sigma_{4}-\sigma_{31})}  \\
    &= \Phi_{\sigma_{31}d}\left(a_{4}a^{-1}_{31}\right)y^{\sigma_{2}d},
\end{align*}
\endgroup
by \cref{prop:non-len p results 2}. Our new term is now
\begin{align*}
    \phi(\Phi_{\delta}(a_{2}))\cdot \overline{u} \cdot \Phi_{\delta}\left(a^{-1}_{2}\right) &=_{G} \phi(\Phi_{\delta}(a_{2}))y^{\alpha+\lambda(1-\e_{y})}\\
    &=_{G} \Phi_{\sigma_{31}d}\left(a_{4}a^{-1}_{31}\right)y^{\alpha+\lambda(1-\e_{y})+\sigma_{2}d}. 
\end{align*}
Now we apply a further $x$-shift to move the $a^{-1}_{31}$ component, to obtain $v$ as required. We first note that $\sigma = \sigma_{31}-\sigma_{2}$, since $a_{3} = a_{2}^{-1}a_{31}$, and so
\[ d(\sigma_{31} -\sigma_{2}) - \alpha_{1} - \lambda_{1}(1-\e_{y}) = \sigma d -\alpha_{1}-\lambda_{1}(1-\e_{y}) = \delta.
\]
We apply the following $x$-shift, again using \cref{prop:non-len p results 2}:
\begingroup
\addtolength{\jot}{0.5em}
\begin{align*}
    \phi(\Phi_{\delta}(a_{31}))^{-1}\cdot \Phi_{\sigma_{31}d}\left(a_{4}a^{-1}_{31}\right)y^{\alpha+\lambda(1-\e_{y})+\sigma_{2}d}\cdot \Phi_{\delta}(a_{31}) &=_{G} \phi(\Phi_{\delta}(a_{31}))^{-1}\Phi_{\sigma_{31}d}(a_{4})y^{\alpha+\lambda(1-\e_{y})+\sigma_{2}d} \\
    &=_{G} \left[\phi\left(\Phi_{\delta}\left(a^{-1}_{31}\right)\right)\right]_{F}y^{-\sigma_{31}d}\Phi_{\sigma_{31}d}(a_{4})y^{\alpha+\lambda(1-\e_{y})+\sigma_{2}d} \\
    &=_{G} \left[\phi\left(\Phi_{\delta}\left(a^{-1}_{31}\right)\right)\right]_{F}a_{4}y^{\alpha+\lambda(1-\e_{y})+\sigma_{2}d-\sigma_{3}d} \\
    &= \Phi_{\e_{y}\delta}\left(\left[\phi\left(a^{-1}_{31}\right)\right]_{F}\right)a_{4}y^{-\delta}y^{\alpha_{2}+\lambda_{2}(1-\e_{y})} \\
    &= v. 
\end{align*}
\endgroup
We now have a route from $u$ to $v$ via $x$-shifts and $y$-shifts only. 

Suppose instead that $\Phi_{\e_{y}\delta - \sigma_{31}d}\left(\left[\phi(a_{2})\right]_{F}\right)$ fully cancels. Then we have \\
$a_{31} = \left(\Phi_{\e_{y}\delta - \sigma_{31}d}\left(\left[\phi(a_{2})\right]_{F}\right)\right)^{-1}a_{32}$, which implies that
\[ \phi(a_{31})^{-1} = \phi(a_{32})^{-1}\phi\left(\Phi_{\e_{y}\delta - \sigma_{31}d}\left(\left[\phi(a_{2})\right]_{F}\right)\right).
\]
For notation, let $\overline{a_{2}} = \phi\left(\Phi_{\e_{y}\delta - \sigma_{31}d}\left(\left[\phi(a_{2})\right]_{F}\right)\right)$, and let $\sigma_{32}$ denote the exponent sum of $a_{32}$. By \cref{prop:non len p 1} and \cref{lemma:1}, we have
\begingroup
\addtolength{\jot}{0.5em}
\begin{align*}
    v_{1} =_{G} \Phi_{\e_{y}\delta}\left(\left[\phi\left(a^{-1}_{31}\right)\right]_{F}\right)a_{32} &=  \Phi_{\e_{y}\delta}\left(\left[\phi\left(a^{-1}_{32}\right)\overline{a_{2}}\right]_{F}\right)a_{32}\\
    &= \Phi_{\e_{y}\delta}\left(\left[\phi\left(a^{-1}_{32}\right)\right]_{F}\right)\Phi_{\e_{y}\delta + \sigma_{32}d}\left(\left[\overline{a_{2}}\right]_{F}\right)a_{32} \\
    &= \left[\phi\left(\Phi_{\delta}\left(a^{-1}_{32}\right)\right)\right]_{F}\Phi_{\e_{y}\delta + \sigma_{32}d}\left(\left[\overline{a_{2}}\right]_{F}\right)a_{32},
\end{align*}
\endgroup
and the result follows by reverse induction on the length of $w$. The case where $a_{2}$ fully cancels with $a_{1}$ follows a similar proof. 
\end{proof}

\subsection{Finite set of representatives}
Again for this section we assume $\phi \in \mathrm{Out}(G(m))$ is known and of the form in \cref{even:all auto forms}.
\begin{defn}
    Let $u \in G_{\mathrm{mod}}$ be $\phi$-CR. We define $\mathcal{D}_{u}$ to be the set of all words $v \in G_{\mathrm{mod}}$ obtained from $u$ via a sequence of $x$-shifts and $y$-shifts. Note that $u \sim_{\phi} v$ for all $v \in \mathcal{D}_{u}$ by \cref{thm:non len p main result}.
\end{defn}
Suppose $u \in G_{\mathrm{mod}}$ is $\phi$-CR. In many cases, the set $\mathcal{D}_{u}$ of all words $v \in G_{\mathrm{mod}}$ obtained from $u$ via a sequence of $x$-shifts and $y$-shifts is infinite. The aim of this section is to show that $\mathcal{D}_{u}$ is precisely an infinite union of finite sets of equal size, which we denote by $\overline{\mathcal{D}_{u_{i}}}$ ($i \in \Z$). We will show that for all $j \neq i$, there exists a fixed $z \in \Z$ such that for all $v_{i} \in \overline{\mathcal{D}_{u_{i}}}$, there exists $v_{j} \in \overline{\mathcal{D}_{u_{j}}}$ such that the quotient elements of $v_{i}$ and $v_{j}$ are equal, and the Garside powers of $v_{i}$ and $v_{j}$ are equal up to a power of $y^{z}$. This will allow us to construct an algorithm to check when two elements are twisted conjugate, with respect to a given $\phi \in \mathrm{Out(G(m))}$.
\comm{
We will use our modular normal form to construct one of these finite sets $\overline{\mathcal{D}_{u_{i}}}$, and the corresponding power $z \in \Z$. Then, when checking if an element $v \in G_{\mathrm{mod}}$ is twisted conjugate to $u$ reduces to checking if there exists $w \in \overline{\mathcal{D}_{u_{i}}}$, such that the quotient element of $v$ and $w$ are the same, and the Garside powers of $v$ and $w$ are equal up to a power of $y^{z}$. }

We first return to our definition of $x$-shifts (recall \cref{defn:x shift}), and distinguish between moving letters from the front to the back of the free component of a word $u \in G_{\mathrm{mod}}$, or vice versa.

\begin{defn}
    Let $u = (u_{11}u_{12}y^{\alpha_{1}}, y^{\alpha_{2}}) \in G_{\mathrm{mod}}$ be $\phi$-CR, and let $\sigma_{11}, \sigma_{12}$ denote the exponent sums of $u_{11}$ and $u_{12}$ respectively. Suppose $v \in G_{\mathrm{mod}}$ is obtained from $u$ via an $x$-shift. We say $v$ is related to $u$ by a \emph{back-to-front} (BF) $x$-shift if $v$ is equivalent to 
    \[ v =_{G} \phi(\Phi_{-\alpha_{1}}(u_{12}))u_{11}y^{\alpha_{1}}y^{\alpha_{2}} =_{G} [\phi(\Phi_{-\alpha_{1}}(u_{12}))]_{F}\Phi_{\sigma_{2}d}(u_{11})y^{\alpha_{1}+\alpha_{2}+\sigma_{12}d}.
    \]
    Similarly $v$ is related to $u$ by a \emph{front-to-back} (FB) $x$-shift if $v$ is equivalent to
    \[ v =_{G} u_{12}y^{\alpha_{1}}\phi^{-1}(u_{11})y^{\alpha_{2}} =_{G} u_{12}\Phi_{\alpha_{1}}([\phi^{-1}(u_{11})]_{F})y^{\alpha_{1}+\alpha_{2}-\e_{x}\e_{y}\sigma_{11}d }.
    \]
    These formulae are obtained using \cref{prop:non-len p results 2}. We let $\xmapsto{BF}$ (resp. $\xmapsto{FB}$) denote a BF (resp. FB) $x$-shift of a generator $x_{i} \in X_{n}$, and we let $\xmapsto{BF*}$ (resp. $\xmapsto{FB*}$) denote a sequence of BF (resp. FB) $x$-shifts.
\end{defn}
We now highlight some cases where $\mathcal{D}_{u}$ is finite. 

\begin{prop}\label{prop:when set finite}
    Let $u = (u_{1}y^{\alpha_{1}}, y^{\alpha_{2}}) \in G_{\mathrm{mod}}$ be $\phi$-CR, and let $\sigma$ denote the exponent sum of $u_{1}$. Suppose $\e_{y} = 1$ and either $\sigma = 0$ or $\e_{x} = -1$. Then the set $\mathcal{D}_{u}$ is finite.
\end{prop}

\begin{proof}
    Since $u$ is $\phi$-CR, the length of the free components of all elements $v \in \mathcal{D}_{u}$ will equal the length of $u_{1}$. For $\mathcal{D}_{u}$ to be finite, we need to show that the exponents of the $y$ terms are bounded in $\mathcal{D}_{u}$. By \cref{defn: len p y shift}, if $\e_{y}=1$, then any $y$-shift leaves the $y$ exponent of $u$ unchanged. It remains therefore to check any $x$-shifts obtained from $u$.
    
    Let $u_{1} = x^{p_{1}}_{i_{1}}\dots x^{p_{q}}_{i_{q}} \in F_{n}$, and consider the $y$ exponent of words after applying BF $x$ shifts to $u$. The exponent follows the following sequence:
    \begin{align*}
    \alpha_{1} + \alpha_{2} \xmapsto{BF} \alpha_{1} + \alpha_{2}+p_{q}d &\xmapsto{BF} \alpha_{1} + \alpha_{2} + (p_{q} + p_{q-1})d \\
    & \quad \vdots \\
    &\xmapsto{BF} \alpha_{1} + \alpha_{2} + \sigma d \\
    &\xmapsto{BF} \alpha_{1} + \alpha_{2} + \sigma d + \e_{x}p_{q}d \\
        &\xmapsto{BF} \alpha_{1} + \alpha_{2} + \sigma d + \e_{x}d(p_{q} + p_{q-1}) \\
        & \quad \vdots \\
        &\xmapsto{BF} \alpha_{1} + \alpha_{2} + \sigma d(1+\e_{x}) \\
        & \quad \vdots 
    \end{align*}
    Therefore if $\sigma = 0$ or $\e_{x} = -1$, our $y$ exponent must return to $\alpha_{1} + \alpha_{2}$ and repeat the same pattern. A similar sequence holds if we consider FB shifts, and so the set $\mathcal{D}_{u}$ is finite.
\end{proof}
\begin{exmp}
    Returning to \cref{exmp:non len p algorithm}, we compute the set $\mathcal{D}_{u}$. We note that by \cref{prop:nonlen p cyc reduced}, $u$ is $\phi$-CR. Moreover, since the exponent sum of the free component of $u$ equals 0, then the set $\mathcal{D}_{u}$ will be finite by \cref{prop:when set finite}. By computing all possible $x$-shifts and $y$-shifts from $u$, we find that $\mathcal{D}_{u}$ contains the following six elements.
    \[ \mathcal{D}_{u} = 
    \begin{Bmatrix}
        \left(x_{0}x^{-1}_{2}y^{2}, 1_{G}\right), & \left(x_{1}x_{0}^{-1}y^{2}, 1_{G}\right), & \left(x_{2}x^{-1}_{1}y^{2}, 1_{G}\right), \\
        \left(x^{-1}_{2}x_{1}y, y^{-3}\right), & \left(x^{-1}_{0}x_{2}y, y^{-3}\right), & \left(x^{-1}_{1}x_{0}y, y^{-3}\right).
    \end{Bmatrix}
    \]
    \cref{table:Au} summarises how each of the elements in $\mathcal{D}_{u}$ are linked via $x$-shifts and $y$-shifts. Here blue lines represent $y$-shifts, and red lines represent $x$-shifts. 
    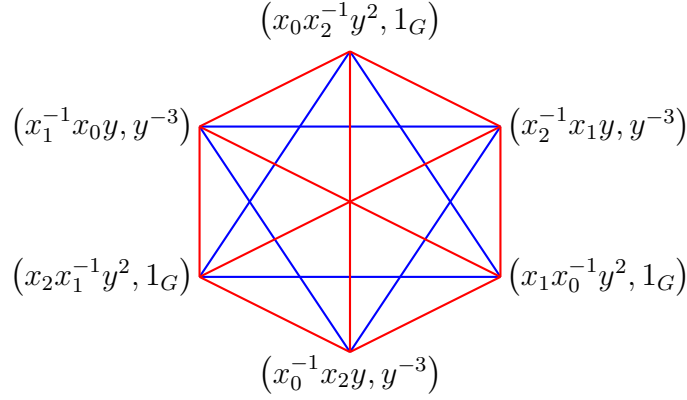
\begin{figure}[!h]
        \centering
        \begin{tikzpicture}
            \node at (0,2.4) {$\left(x_{0}x^{-1}_{2}y^{2}, 1_{G}\right)$};
            \node at (0, -2.3) {$\left(x^{-1}_{0}x_{2}y, y^{-3}\right)$};
            \node at (3.3, 1) {$\left(x^{-1}_{2}x_{1}y, y^{-3}\right)$};
            \node at (3.3, -1) {$\left(x_{1}x^{-1}_{0}y^{2}, 1_{G}\right)$};
            \node at (-3.3, 1) {$\left(x^{-1}_{1}x_{0}y, y^{-3}\right)$};
            \node at (-3.3, -1) {$\left(x_{2}x^{-1}_{1}y^{2}, 1_{G}\right)$};
            \draw[thick, blue] (0,2) -- (2,-1); 
            \draw[thick, blue] (0,2) -- (-2,-1);
            \draw[thick, blue] (-2,-1) -- (2,-1);
            \draw[thick, blue] (-2,1) -- (2,1);
            \draw[thick, blue] (-2, 1) -- (0,-2);
            \draw[thick, blue] (2,1) -- (0,-2);
            \draw[thick, red] (0,2) -- (2,1) ;
            \draw[thick, red] (2,-1) -- (2,1) ;
            \draw[thick, red] (2,-1) -- (0,-2) ;
            \draw[thick, red] (0,-2) -- (-2,-1) ;
            \draw[thick, red] (-2,-1) -- (-2,1) ;
            \draw[thick, red] (-2,1) -- (0,2) ;
            \draw[thick, red] (0,2) -- (0,-2) ;
            \draw[thick, red] (-2,1) -- (2,-1) ;
            \draw[thick, red] (-2,-1) -- (2,1) ;
        \end{tikzpicture}
        \caption{Connections between elements of $\mathcal{D}_{u}$}
        \label{table:Au}
    \end{figure}
\end{exmp}
We now establish some results about $x$-shifts and $y$-shifts, which will allow us to construct our finite set of representatives $\overline{\mathcal{D}_{u_{i}}}$ for every twisted conjugacy class.

\begin{lemma}\label{lemma:BF to same elements}
    Let $v = (v_{1}, v_{2}), w = (w_{1}, w_{2}) \in G_{\mathrm{mod}}$ be $\phi$-CR. Suppose $v \xmapsto{BF} u_{v} = (u_{v_{1}}, u_{v_{2}}) \in G_{\mathrm{mod}}$ and $w \xmapsto{BF} u_{w} = (u_{w_{1}}, u_{w_{2}}) \in G_{\mathrm{mod}}$, such that $u_{v_{1}}= u_{w_{1}}$. Then $v_{1} = w_{1}$. 
\end{lemma}

\begin{proof}
    Let $v = \left(\overline{v_{1}}x^{\varepsilon_{i}}_{i}y^{\alpha_{1}}, y^{\alpha_{2}}\right)$ and $w = \left(\overline{w_{1}}x^{\varepsilon_{j}}_{j}y^{\lambda_{1}}, y^{\lambda_{2}}\right)$, for some $\overline{v_{1}}, \overline{w_{1}} \in F_{n}$, \\$i,j \in \{0, \dots, n-1\}, \varepsilon_{i}, \varepsilon_{j} = \pm 1$. If we apply a BF $x$-shift to each of these words, we get
    \begin{align*}
        v &\xmapsto{BF} \left[\phi\left(\Phi_{-\alpha_{1}}\left(x^{\varepsilon_{i}}_{i}\right)\right)\right]_{F}\Phi_{\varepsilon_{i}d}\left(\overline{v_{1}}\right)y^{\alpha_{1}+\varepsilon_{i}d}y^{\alpha_{2}}, \\
        w &\xmapsto{BF} \left[\phi\left(\Phi_{-\lambda_{1}}\left(x^{\varepsilon_{j}}\right)\right)\right]_{F}\Phi_{\varepsilon_{j}d}\left(\overline{w_{1}}\right)y^{\lambda_{1}+\varepsilon_{j}d}y^{\lambda_{2}}. 
    \end{align*}
    To obtain equality of the quotient elements of $v$ and $w$, the powers of $\left[\phi\left(\Phi_{-\alpha_{1}}\left(x^{\varepsilon_{i}}_{i}\right)\right)\right]_{F}$ and $\left[\phi\left(\Phi_{-\lambda_{1}}\left(x^{\varepsilon_{j}}_{j}\right)\right)\right]_{F}$ must be equal. By $(vi)$ of \cref{prop:non-len p results 2}, this implies that $\e_{x}\varepsilon_{i} = \e_{x}\e_{j}$, and so $\e_{i}=\e_{j}$. This immediately implies that $\overline{v_{1}} = \overline{w_{1}}$, $\alpha_{1} = \lambda_{1}$ and $i=j$ as required. 
\end{proof}

\begin{cor}\label{cor:existence of chain BF moves}
    For every $\phi$-CR element $u = (u_{1}, u_{2}) \in G_{\mathrm{mod}}$, there exists a finite sequence of BF $x$-shifts to an element $v = (v_{1}, v_{2}) \in G_{\mathrm{mod}}$, such that $u_{1} = v_{1}$. 
\end{cor}

\begin{proof}
    Let $u_{i} = (u_{i_{1}}, u_{i_{2}}) \in G_{\mathrm{mod}}$ denote elements obtained from $u = u_{0}$ via a sequence of BF $x$-shifts, where $i \in \Z_{\geq 0}$. Since $u$ is $\phi$-CR, there exists $i \leq j$ such that $u_{i_{1}} = u_{j_{1}}$, that is, we return to an element with the same quotient element. If $i = 0$ we are done, so suppose $i \neq 0$. \cref{lemma:BF to same elements} implies that $u_{(i-1)_{1}} = u_{(j-1)_{1}}$, and so we can reduce the length of our sequence. We repeat this procedure, until we are left with a sequence such that $u_{1_{1}} = u_{s_{1}}$, for some $s < j$. Again we then have $u_{0_{1}} = u_{(s-1)_{1}}$, which gives us the finite sequence required.
\end{proof}

\begin{prop}\label{prop: commute x shifts y shifts}
    BF $x$-shifts and $y$-shifts commute. In particular, if $u \in G_{\mathrm{mod}}$ is $\phi$-CR and $u \xmapsto{BF*} v_{1} \xleftrightarrow{y} w$, for some $v_{1}, w \in G_{\mathrm{mod}}$, then there exists $v_{2} \in G_{\mathrm{mod}}$, such that $u \xleftrightarrow{y} v_{2} \xmapsto{BF*} w$.
\end{prop}

\begin{proof}
    Let $u = (u_{11}u_{12}y^{\alpha_{1}}, y^{\alpha_{2}}) \in G_{\mathrm{mod}}$ be $\phi$-CR, and let $\sigma_{12}$ denote the exponent sum of $u_{12}$. Suppose we first apply a sequence of BF $x$-shifts, followed by a $y$-shift. This gives us, for some $\lambda \in \Z$,
    \begingroup
    \addtolength{\jot}{0.5em}
    \begin{align*}
        u &\xmapsto{BF*} \left[\phi\left(\Phi_{-\alpha_{1}}\left(u_{12}\right)\right)\right]_{F}\Phi_{\sigma_{12}d}(u_{11})y^{\alpha_{1}+\alpha_{2}+\sigma_{12}d} \\
        &\xleftrightarrow{y} \Phi_{-\e_{y}\lambda}\left([\phi\left(\Phi_{-\alpha_{1}}\left(u_{12}\right)\right)\right]_{F}\Phi_{\sigma_{12}d}(u_{11}))y^{\alpha_{1}+\alpha_{2}+\sigma_{12}d+\lambda(1-\e_{y})}.
    \end{align*}
    \endgroup
    Suppose instead that we first apply a $y$-shift with respect to the same $\lambda \in \Z$. This gives us
    \[ u \xleftrightarrow{y} \Phi_{-\e_{y}\lambda}(u_{11}u_{12})y^{\alpha_{1}+\lambda(1-\e_{y})}y^{\alpha_{2}}.
    \]
    Now we can apply a sequence of BF $x$-shifts to get
    \begingroup
    \addtolength{\jot}{0.5em}
    \begin{align*}
        u \xleftrightarrow{y} \Phi_{-\e_{y}\lambda}(u_{11}u_{12})y^{\alpha_{1}+\lambda(1-\e_{y})}y^{\alpha_{2}} &\xmapsto{BF*} \phi\left(\Phi_{-\alpha_{1}-\lambda(1-\e_{y})}\left(\Phi_{-\e_{y}\lambda}\left(u_{12}\right)\right)\right)\Phi_{-\e_{y}\lambda}(u_{11})y^{\alpha_{1}+\alpha_{2}+\lambda(1-\e_{y})}\\
        &= \phi\left(\Phi_{-\alpha_{1}-\lambda}\left(u_{12}\right)\right)\Phi_{-\e_{y}\lambda}(u_{11})y^{\alpha_{1}+\alpha_{2}+\lambda(1-\e_{y})}\\
        &=_{G} \left[\phi\left(\Phi_{-\alpha_{1}-\lambda}\left(u_{12}\right)\right)\right]_{F}y^{\sigma_{12}d}\Phi_{-\e_{y}\lambda}(u_{11})y^{\alpha_{1}+\alpha_{2}+\lambda(1-\e_{y})}\\
        &=_{G} \left[\phi\left(\Phi_{-\alpha_{1}-\lambda}\left(u_{12}\right)\right)\right]_{F}\Phi_{-\e_{y}\lambda + \sigma_{12}d}(u_{11})y^{\alpha_{1}+\alpha_{2}+ \sigma_{12}d +\lambda(1-\e_{y})} \\
        &=_{G} \Phi_{-\e_{y}\lambda}\left(\left[\phi\left(\Phi_{-\alpha_{1}}\left(u_{12}\right)\right)\right]_{F}\right)\Phi_{-\e_{y}\lambda + \sigma_{12}d}(u_{11})y^{\alpha_{1}+\alpha_{2}+ \sigma_{12}d +\lambda(1-\e_{y})} \\
        &= \Phi_{-\e_{y}\lambda}\left(\left[\phi\left(\Phi_{-\alpha_{1}}\left(u_{12}\right)\right)\right]_{F}\Phi_{\sigma_{12}d}(u_{11})\right)y^{\alpha_{1}+\alpha_{2}+\sigma_{12}d+\lambda(1-\e_{y})},
    \end{align*}  
    \endgroup
    using \cref{prop:non-len p results 2}.
\end{proof}

\begin{prop}\label{prop:FB not needed}
    Let $u,v \in G_{\mathrm{mod}}$ be $\phi$-CR, such that $u \xmapsto{BF*} v$. Then there exists a sequence of FB $x$-shifts $v \xmapsto{FB*} w$, for some $w \in G_{\mathrm{mod}}$, such that $u \xleftrightarrow{y} w$.
\end{prop}
\begin{proof}
    Let $u = (u_{11}u_{12}y^{\alpha_{1}}, y^{\alpha_{2}}) \in G_{\mathrm{mod}}$. By assumption we have \\ $v = \left[\phi\left(\Phi_{-\alpha_{1}}\left(u_{12}\right)\right)\right]_{F}\Phi_{\sigma_{12}d}(u_{11})y^{\alpha_{1}+\alpha_{2}+\sigma_{12}d}$, where $\sigma_{12}$ denotes the exponent sum of $u_{12}$. We can then apply the following sequence of FB $x$-shifts:
    \begingroup
    \addtolength{\jot}{0.5em}
    \begin{align*}
        v = \left[\phi\left(\Phi_{-\alpha_{1}}\left(u_{12}\right)\right)\right]_{F}\Phi_{\sigma_{12}d}(u_{11})y^{\alpha_{1}+\alpha_{2}+\sigma_{12}d} &\xmapsto{FB*} \Phi_{\sigma_{12}d}(u_{11})y^{\alpha_{1}+\sigma_{12}d}\phi^{-1}\left(\left[\phi\left(\Phi_{-\alpha_{1}}\left(u_{12}\right)\right)\right]_{F}\right)y^{\alpha_{2}} \\
        &=_{G} \Phi_{\sigma_{12}d}(u_{11})y^{\alpha_{1}+\sigma_{12}d}\Phi_{-\alpha_{1}}(u_{12})y^{-\e_{y}\sigma_{12}d}y^{\alpha_{2}} \\
        &=_{G} \Phi_{\sigma_{12}d}(u_{11}u_{12})y^{\alpha_{1}+\sigma_{12}d(1-\e_{y})}y^{\alpha_{2}},
    \end{align*}
    \endgroup
    using \cref{prop:non-len p results 2}. This can then be transformed back to $u$ by the following $y$-shift:
    \begingroup
    \addtolength{\jot}{0.5em}
    \begin{align*}
        \Phi_{\sigma_{12}d}(u_{11}u_{12})y^{\alpha_{1}+\sigma_{12}d(1-\e_{y})}y^{\alpha_{2}} &\xleftrightarrow{y} \Phi_{-\sigma_{12}d}(\Phi_{\sigma_{12}d}(u_{11}u_{12}))y^{\alpha_{1}+\sigma_{12}d(1-\e_{y})+\sigma_{12}d(\e_{y}-1)}y^{\alpha_{2}} \\
        &= u_{11}u_{12}y^{\alpha_{1}}y^{\alpha_{2}}. \tag*{\qedhere}
    \end{align*}
    \endgroup
\end{proof}
\begin{rmk}\label{rmk:construction finite set D}
    We are now able to compute our finite set of representatives for each twisted conjugacy class as follows. 

    Starting with $u_{0} = \left(u_{0_{1}}, u_{0_{2}}\right) \in G_{\mathrm{mod}}$ which is $\phi$-CR, compute a sequence of BF $x$-shifts until we reach an element $u_{i} = \left(u_{i_{1}}, u_{i_{2}}\right) \in G_{\mathrm{mod}}$, for some $i \in \Z_{\geq 0}$, such that $u_{0_{1}} = u_{i_{1}}$ (existence of this sequence follows from \cref{cor:existence of chain BF moves}). In particular, if we compute further BF $x$-shifts from $u_{i}$, we would obtain words with the same quotient elements as $\{u_{0}, u_{1}, \dots, u_{i-1}\}$.
     
    Then, for each element $u_{j} = \left(u_{j_{1}}, u_{j_{2}}\right) \in G_{\mathrm{mod}}$ in this sequence, where $0 \leq j \leq i-1$, we compute $y$-shifts by twisted conjugating by $y^{k}$, for each $k \in \{0, \dots, n-1\}$, to obtain elements of the form $u_{j_{k}} = \left(u_{j_{k_{1}}}, u_{j_{k_{2}}}\right) \in G_{\mathrm{mod}}$. Here if we were to compute further $y$-shifts, we would obtain words with the same quotient elements as $\{u_{j_{o}}, u_{j_{1}}, \dots, u_{j_{k-1}}\}$ for each $j$. We then define our finite set of representatives to be
    \[ \overline{\mathcal{D}_{u_{0}}} = \bigcup^{n-1}_{k=0}\bigcup^{i-1}_{j=0} u_{j_{k}}.
    \]
    This set is precisely all words obtained from $u_{0}$ via BF $x$-shifts and $y$-shifts, such that any further computations give us words which have the same quotient element as words in $\overline{\mathcal{D}_{u_{0}}}$. Note no further computations are required, since BF $x$-shifts and $y$-shifts commute by \cref{prop: commute x shifts y shifts}, and FB $x$-shifts are also not required by \cref{prop:FB not needed}.
    
    Now suppose $u_{0_{2}} = y^{s}$ and $u_{i_{2}} = y^{t}$, for some $s,t \in \Z$. We define the \emph{twisted shift} of $u_{0}$ to be the element $y^{z} \in \Z$, where $z = |s-t|$. In particular, we can write 
    \[ \mathcal{D}_{u} = \bigcup_{q \in \Z} \overline{\mathcal{D}_{u_{q}}},
    \]
    where each set $\overline{\mathcal{D}_{u_{q}}}$ is equal to $\overline{\mathcal{D}_{u_{0}}}$ up to the twisted shift. In other words, for every $q \in \Z$ there exists a fixed $\lambda \in \Z$, such that for every element $v = (v_{1}, v_{2}) \in \overline{\mathcal{D}_{u_{q}}}$, there exists $w = (w_{1}, w_{2}) \in \overline{\mathcal{D}_{u_{0}}}$ such that $v_{1} = w_{1}$ and $v_{2} = w_{2}y^{\lambda z}$.
\end{rmk}

\begin{exmp}\label{exmp:infinite set}
    We use \cref{exmp:non len p algorithm} to illustrate how to construct the set $\overline{\mathcal{D}_{u_{0}}}$ as defined in \cref{rmk:construction finite set D}. Here we consider the $\phi$-CR element $u_{0} = \left(x_{0}x_{2}y^{2}, 1_{G}\right) \in G_{\mathrm{mod}}$. \cref{fig:enter-label} describes the set $\overline{\mathcal{D}_{u_{0}}}$ of $x$-shifts and $y$-shifts from $u$, such that further shifts produce the same elements up to the twisted shift, which is $y^{24}$. Here the blue arrows represent BF $x$-shifts of single letters, and the red arrows represent $y$-shifts. \\
    \begin{figure}[h]
    \centering
        \begin{tikzpicture}
            \node at (0,0) {$(x_{0}x_{2}y^{2}, 1_{G})$};
            \node at (0, -1) {$(x_{1}x_{2}, y^{6})$};
            \draw[thick, blue, ->] (0,-0.3) -- (0,-0.7);
            \node at (4, 0) {$(x_{1}x_{0}y^{2}, 1_{G})$};
            \node at (8,0) {$(x_{2}x_{1}y^{2}, 1_{G})$};
            \draw[thick, red, <->] (1.2, 0) -- (2.8,0);
            \node at (4, -1) {$(x_{2}x_{0}, y^{6})$};
            \node at (8, -1) {$\left(x_{0}x_{1}, y^{6}\right)$};
            \node at (0, -2) {$(x_{2}x_{0}y, y^{9})$};
            \node at (4, -2) {$(x_{0}x_{1}y, y^{9})$};
            \node at (8, -2) {$(x_{1}x_{2}y, y^{9})$};
            \node at (0, -3) {$(x_{1}x_{1}y^{2}, y^{12})$};
            \node at (4, -3) {$(x_{2}x_{2}y^{2}, y^{12})$};
            \node at (8, -3) {$(x_{0}x_{0}y^{2}, y^{12})$};
            \node at (0, -4) {$(x_{0}x_{0},y^{18})$};
            \node at (4, -4) {$(x_{1}x_{1},y^{18})$};
            \node at (8, -4) {$(x_{2}x_{2},y^{18})$};
            \node at (0, -5) {$(x_{0}x_{2}y, y^{21})$};
            \node at (4, -5) {$(x_{1}x_{0}y, y^{21})$};
            \node at (8, -5) {$(x_{2}x_{1}y, y^{21})$};
            \draw[thick, red, <->] (5.2, 0) -- (6.8,0);
            \draw[thick, red, <->] (1.2, -1) -- (2.8,-1);
            \draw[thick, red, <->] (5.2, -1) -- (6.8, -1);
            \draw[thick, red, <->] (1.2, -2) -- (2.8,-2);
            \draw[thick, red, <->] (5.2, -2) -- (6.8, -2);
            \draw[thick, red, <->] (1.2, -3) -- (2.8,-3);
            \draw[thick, red, <->] (5.2, -3) -- (6.8, -3);
            \draw[thick, red, <->] (1.2, -4) -- (2.8,-4);
            \draw[thick, red, <->] (5.2, -4) -- (6.8, -4);
            \draw[thick, red, <->] (1.2, -5) -- (2.8,-5);
            \draw[thick, red, <->] (5.2, -5) -- (6.8, -5);
            \draw[thick, blue, ->] (4,-0.3) -- (4,-0.7);
            \draw[thick, blue, ->] (8,-0.3) -- (8,-0.7);
            \draw[thick, blue, ->] (0,-1.3) -- (0,-1.7);
            \draw[thick, blue, ->] (0,-2.3) -- (0,-2.7);
            \draw[thick, blue, ->] (0,-3.3) -- (0,-3.7);
            \draw[thick, blue, ->] (0,-4.3) -- (0,-4.7);
            \draw[thick, blue, ->] (4,-1.3) -- (4,-1.7);
            \draw[thick, blue, ->] (4,-2.3) -- (4,-2.7);
            \draw[thick, blue, ->] (4,-3.3) -- (4,-3.7);
            \draw[thick, blue, ->] (4,-4.3) -- (4,-4.7);
            \draw[thick, blue, ->] (8,-1.3) -- (8,-1.7);
            \draw[thick, blue, ->] (8,-2.3) -- (8,-2.7);
            \draw[thick, blue, ->] (8,-3.3) -- (8,-3.7);
            \draw[thick, blue, ->] (8,-4.3) -- (8,-4.7);
        \end{tikzpicture}
        \caption{Connections between elements of $\overline{\mathcal{D}_{u_{0}}}$}
        \label{fig:enter-label}
    \end{figure}
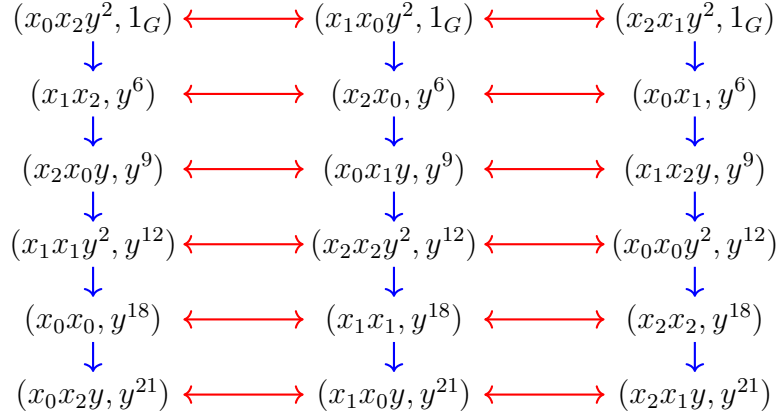\\
    For example, if we compute a BF $x$-shift of $(x_{0}x_{2}y, y^{21})$, we obtain $(x_{0}x_{2}y^{2}, y^{24})$. This is the same element as $(x_{0}x_{2}y^{2}, 1_{G})$, up to the twisted shift.
\end{exmp}
Let $u, v \in G_{\mathrm{mod}}$ be $\phi$-CR. By \cref{thm:non len p main result}, $u \sim_{\phi} v$ if and only if $\mathcal{D}_{u} = \mathcal{D}_{v}$. By \cref{rmk:construction finite set D}, this holds if and only if $v \in \overline{\mathcal{D}_{u_{q}}}$, for some $q \in \Z$. 

\begin{cor}\label{cor:Step 4 algorithm}
    Let $u = (u_{1}, u_{2}), v = (v_{1}, v_{2}) \in G_{\mathrm{mod}}$ be $\phi$-CR. Let $\overline{\mathcal{D}_{u}}$ be the finite set of $x$-shifts and $y$-shifts from $u$, as defined in \cref{rmk:construction finite set D}. Let $y^{z}$, for some $z \in \Z$, be the twisted shift of $u$. Then $u \sim_{\phi} v$ if and only if there exists $w = (w_{1}, w_{2}) \in \overline{\mathcal{D}_{u}}$ such that $w_{1} = v_{1}$ and $v_{2} = w_{2}y^{\lambda z}$, for some $\lambda \in \Z$. 
\end{cor}
We can now construct an algorithm to solve the twisted conjugacy problem for even dihedral Artin groups, when $\phi \in \mathrm{Out}(G(m))$ is known.
\begin{prop}\label{cor:even cases main result}
    The $\mathrm{TCP}_{\phi}(G(m))$ is solvable, when $m$ is even, for all outer automorphisms of the form
    \[ \phi \colon x_{0} \mapsto x^{\e_{x}}_{0}y^{d}, \; y \mapsto y^{\e_{y}},
\]
where $\e_{x}, \e_{y} \in \{\pm 1\}$ and $d \in \Z$.
\end{prop}

\begin{proof}
    Our algorithm is as follows:
    
    \textbf{Input:}
    \begin{enumerate}
        \item $m \in \Z_{\geq 2}$ which is even.
        \item Words $u,v \in X^{\ast}$ representing group elements.
        \item $\phi \in \mathrm{Out}(G(m))$ as in \cref{even:all auto forms}.
    \end{enumerate}
    \textbf{Step 1: Modular normal forms}
    \begin{adjustwidth}{1.5cm}{}
        Write $u,v$ in modular normal forms $u_{m} = (u_{1}, u_{2}), v_{m} = (v_{1}, v_{2}) \in G_{\mathrm{mod}}$ (see \cref{defn:mod normal form}). 
    \end{adjustwidth}
    \textbf{Step 2: Compare $y$ powers}
    \begin{adjustwidth}{1.5cm}{}
        We can only apply this step when $d=0 \Mod{n}$. Check the conditions from \cref{cor:len p compare y} on $u_{2}$ and $v_{2}$. If these conditions are met, \textbf{Output} = \texttt{\color{blue}False}.
    \end{adjustwidth}
    \textbf{Step 3: Twisted cyclic reduction}
    \begin{adjustwidth}{1.5cm}{}
        Apply $x$-shifts to $u_{m}$ and $v_{m}$, to obtain $\phi$-CR elements $\overline{u}_{m} = (\overline{u}_{m_{1}}, \overline{u}_{m_{2}}) \in G_{\mathrm{mod}}$ and $\overline{v}_{m} = (\overline{v}_{m_{1}}, \overline{v}_{m_{2}}) \in G_{\mathrm{mod}}$ respectively, using the conditions from \cref{prop:nonlen p cyc reduced}. Let $\overline{u_{1}}$ and $\overline{v_{1}}$ denote the free component of $\overline{u}_{m}$ and $\overline{v}_{m}$ respectively. If $l(\overline{u_{1}}) \neq l(\overline{v_{1}})$, then \textbf{Output} = \texttt{\color{blue}False} by \cref{thm:non len p main result}. 
    \end{adjustwidth}
    \textbf{Step 4: Set of representatives}
    \begin{adjustwidth}{1.5cm}{}
        If $\phi$ satisfies the conditions from \cref{prop:when set finite}, then compute the set $\mathcal{D}_{u}$ of all $x$-shifts and $y$-shifts from $\overline{u}_{m}$. If $\overline{v}_{m} \in \mathcal{D}_{u}$, then \textbf{Output} = \texttt{\color{blue}True}. Otherwise, \textbf{Output} = \texttt{\color{blue}False}.
        
        Otherwise, compute the finite set $\overline{\mathcal{D}_{u}}$ of BF $x$-shifts and $y$-shifts from $\overline{u}_{m}$, as defined in \cref{rmk:construction finite set D}. Let $y^{z}$ be the twisted shift of $\overline{u}_{m}$, for some $z \in \Z$. If there exists $w = (w_{1}, w_{2}) \in \overline{\mathcal{D}_{u}}$, such that $w_{1} = \overline{v}_{m_{1}}$ and $w_{2} = \overline{v}_{m_{2}}y^{\lambda z}$, for some $\alpha \in \Z$, then \textbf{Output} = \texttt{\color{blue}True}. Otherwise, \textbf{Output} = \texttt{\color{blue}False}.
    \end{adjustwidth} 
\end{proof}
\begin{rmk}\label{rmk:justify general TCP DAm}
    To generalise this result, we consider $\phi \in \mathrm{Out}(G(m))$ as in \cref{even:all auto forms}, where the parameters $\e_{x}, \e_{y}$ and $d$ are unknown. We can assume $d = kn + c$ for some $k \in \Z$, $c \in \{0, \dots, n-1\}$. To determine if $u \sim_{\phi} v$, for some $u,v \in X^{\ast}$, then we only need to check all $\phi \in \mathrm{Out}(G(m))$ of the form
\[ \phi \colon x_{0} \mapsto x^{\e_{x}}_{0}y^{c}, \; y \mapsto y^{\e_{y}},
\]
where $\e_{x}, \e_{y} \in \{\pm 1\}$ and $c \in \{0, \dots, n-1\}$. To see this, note that if $d = c \Mod{n}$, then the finite set of twisted conjugacy representatives $\overline{\mathcal{D}_{u}}$ will be the same, with respect to $\phi$ with parameter $c$ or $d$, up to a power of the Garside element. This is due to the fact that the quotient element of a word $u \in G_{\mathrm{mod}}$ is preserved under both $x$-shifts and $y$-shifts when $d = 0 \Mod{n}$. 

By checking all possible cases for $\e_{x}, \e_{y}$ and $c \in \{0, \dots, n-1\}$ (i.e. $4n$ cases in total), whilst keeping track of changes to the Garside power for each case, we can determine all possibilities for when $u \sim_{\phi} v$, for some $\phi \in \mathrm{Out}(G(m))$. 
\end{rmk}
We now prove our main result.
\TCP
    \begin{proof}
    By \cite{crowe_twisted_2024}, it remains to prove this result for even dihedral Artin groups. Let $u,v \in X^{\ast}$, and let $\psi \in \mathrm{Inn}(G(m))$, i.e. $\psi(w) = g^{-1}wg$ for some $g \in X^{\ast}$. We need to check two cases. 
\begin{enumerate}
    \item $u \sim_{\psi} v$: Here $v =_{G} \psi(w)^{-1}uw = g^{-1}w^{-1}guw$. Rearranging gives $gv =_{G} w^{-1}(gu)w$, and so this case reduces to solving the conjugacy problem for $(gu, gv)$. The conjugacy problem in dihedral Artin groups is known to be solvable in linear time \cite[Proposition 3.1]{Holt2015}.
    \item $u \sim_{\psi \circ \phi} v$, where $\phi \in \mathrm{Out}(G(m))$ is of the form in \cref{even:all auto forms}: Here $v =_{G} \psi(\phi(w))^{-1}uw = g^{-1}\phi(w)^{-1}guw$. Rearranging gives $gv =_{G} \phi(w)^{-1}(gu)w$, and so this case reduces to solving the $\mathrm{TCP}_{\phi}(G(m))$ with respect to $\phi$. This is solvable by \cref{cor:even cases main result}.
\end{enumerate}
\end{proof}

\subsection{Complexity results}
We are able to show some results about the complexity of our algorithm to solve the twisted conjugacy problem for even dihedral Artin groups. Our first goal is to establish an upper bound for the size of the set $\overline{\mathcal{D}_{u}}$ of representatives generated in Step 4 of our algorithm, based on the length of our input word $u \in G_{\mathrm{mod}}$.

We first prove some results concerning the iteration of free components under $\phi$.
\begin{lemma}\label{lemma:powers of phi free gens}
    Let $i \in \{0, \dots, n-1\}$, $x_{i} \in X_{n}$, $r_{i} \in \{\pm 1\}$. Let
    \[\phi^{k}_{F}\left(x^{r_{i}}_{i}\right) = \underbrace{\phi_{F}\left(\phi_{F}\left(\dots \left(\phi_{F}\left(x^{r_{i}}_{i}\right)\right)\dots\right)\right)}_{k},
    \]
    for some $k \in \Z_{>0}$. 
    \begin{enumerate} 
        \item[(1)] If $\e_{y} = \e_{x} = 1$, then 
        \[ \phi^{2k}_{F}\left(x^{r_{i}}_{i}\right) = \begin{cases}
        x^{r_{i}}_{i}, & r_{i} = 1, \\
        x^{r_{i}}_{[i+2kd]}, & r_{i} = -1.
    \end{cases}
        \]
        \item[(2)] If $\e_{y} = 1, \e_{x} = -1$, then $\phi^{2k}_{F}(x^{r_{i}}_{i}) = x^{r_{i}}_{[i+kd]}$.
        \item[(3)] If $\e_{y} = -1, \e_{x} = 1$, then $\phi^{2k}_{F}(x^{r_{i}}_{i}) = x^{r_{i}}_{i}$.
        \item[(4)] If $\e_{y} = \e_{x} = -1$, then 
        \[ \phi^{2k}_{F}(x^{r_{i}}_{i}) = \begin{cases}
        x^{r_{i}}_{[i+kd]}, & r_{i} = 1,\\
        x^{r_{i}}_{[i-kd]}, & r_{i} = -1.
        \end{cases}
        \]
    \end{enumerate}
\end{lemma}
\begin{proof}
    We prove the first relation, and the remaining cases follow similarly. If $\e_{y} = \e_{x} = 1$, then 
    \[ \phi_{F}(x^{r_{i}}_{i}) = 
    \begin{cases}
        x^{\e_{x}r_{i}}_{i}, & r_{i} = 1, \\
        x^{\e_{x}r_{i}}_{[i+d]}, & r_{i} = -1.
    \end{cases}
    \]
    Then 
    \[ \phi^{2}_{F}(x^{r_{i}}_{i}) = 
    \begin{cases}
        x^{r_{i}}_{i}, & r_{i} = 1, \\
        x^{r_{i}}_{[i+2d]}, & r_{i} = -1.
    \end{cases}
    \]
    This pattern continues to give our result. To guarantee we return to the same power of the free generator, we need to take an even number of iterations of $\phi$. 
\end{proof}

\begin{lemma}\label{lemma:free component complexity 2}
    Let $u = (u_{1}y^{\alpha_{1}}, y^{\alpha_{2}}) \in G_{\mathrm{mod}}$, and let $u_{1} = x^{p_{1}}_{i_{1}}\dots x^{p_{q}}_{i_{q}} \in F_{n}$. Let $k \in \Z_{\geq 0}$. 
    \begin{enumerate}
        \item[(1)] If $\e_{x} = 1$, then 
        \[ \left[\phi^{k}(u_{1})\right]_{F} =_{G} \phi^{k}_{F}\left(x^{p_{1}}_{i_{1}}\right)\Phi_{kdp_{1}}\left(\phi^{k}_{F}\left(x^{p_{2}}_{i_{2}}\right)\right)\dots \Phi_{kd(p_{1}+\dots + p_{q-1})}\left(\phi^{k}_{F}\left(x^{p_{q}}_{i_{q}}\right)\right).
    \]
    \item[(2)] If $\e_{x} = -1$, then 
    \begingroup
    \addtolength{\jot}{0.5em}
    \begin{align*}
        \left[\phi^{2k}(u_{1})\right]_{F} &=_{G} \phi^{2k}_{F}\left(x^{p_{1}}_{i_{1}}\right)\phi^{2k}_{F}\left(x^{p_{2}}_{i_{2}}\right) \dots \phi^{2k}_{F}\left(x^{p_{q}}_{i_{q}}\right), \\
        \left[\phi^{2k+1}(u_{1})\right]_{F} &=_{G}  \phi^{2k+1}_{F}\left(x^{p_{1}}_{i_{1}}\right)\Phi_{p_{1}d}\left(\phi^{2k+1}_{F}\left(x^{p_{2}}_{i_{2}}\right)\right)\dots \Phi_{(p_{1}+\dots + p_{q-1})d}\left(\phi^{2k+1}_{F}\left(x^{p_{q}}_{i_{q}}\right)\right).
    \end{align*}
    \endgroup
    \end{enumerate}
\end{lemma}
\begin{proof}
    Again we prove the first relation, and the second case follows similarly. Let $\sigma$ denote the exponent sum of $u_{1}$. Recall by \cref{prop:non-len p results 2} that
    \[ \phi(u_{1}) =_{G} \phi_{F}\left(x^{p_{1}}_{i_{1}}\right)\Phi_{p_{1}d}\left(\phi_{F}\left(x^{p_{2}}_{i_{2}}\right)\right)\Phi_{(p_{1}+p_{2})d}\left(\phi_{F}\left(x^{p_{3}}_{i_{3}}\right)\right)\dots \Phi_{(p_{1}+p_{2} + \dots + p_{q-1})d}\left(\phi_{F}\left(x^{p_{q}}_{i_{q}}\right)\right)y^{\sigma d}.
    \]
    Note if $\e_{x} = 1$, then $\phi\left(\phi^{k-1}_{F}\left(x^{p_{j}}_{i_{j}}\right)\right) = \phi^{k}_{F}\left(x^{p_{j}}_{i_{j}}\right)y^{p_{j}d}$ for all $i_{j} \in \{0, \dots, n-1\}$, $p_{j} = \pm 1$. Therefore 
    \begingroup
    \addtolength{\jot}{0.5em}
    \begin{align*}
        \phi^{k}(u_{1}) &=_{G} \phi^{k}_{F}\left(x^{p_{1}}_{i_{1}}\right)y^{dp_{1}}\Phi_{(k-1)dp_{1}}\left(\phi^{k}_{F}\left(x^{p_{2}}_{i_{2}}\right)\right)y^{dp_{2}}\Phi_{(k-1)(p_{1}+p_{2})d}\left(\phi^{k}_{F}\left(x^{p_{3}}_{i_{3}}\right)\right)\dots\\
        &\dots y^{dp_{q-1}}\Phi_{(k-1)(p_{1}+p_{2} + \dots + p_{q-1})d}\left(\phi^{k}_{F}\left(x^{p_{q}}_{i_{q}}\right)\right)y^{p_{q}d}y^{(k-1)\sigma d} \\
        &=_{G} \phi^{k}_{F}\left(x^{p_{1}}_{i_{1}}\right)\Phi_{kdp_{1}}\left(\phi^{k}_{F}\left(x^{p_{2}}_{i_{2}}\right)\right)\Phi_{kd(p_{1}+p_{2})}\left(\phi^{k}_{F}\left(x^{p_{3}}_{i_{3}}\right)\right)\dots \Phi_{kd(p_{1}+p_{2} + \dots + p_{q-1})}\left(\phi^{k}_{F}\left(x^{p_{q}}_{i_{q}}\right)\right)y^{k\sigma d}.
    \end{align*}
    \endgroup
    Taking the free component completes the proof. 
\end{proof}

\begin{defn}
    Let $u = (u_{1}y^{\alpha_{1}}, y^{\alpha_{2}}) \in G_{\mathrm{mod}}$ be $\phi$-CR, and let $q$ denote the length of $u_{1}$. We define a \emph{full $x$-shift} of $u$ to be a sequence of $q$ BF $x$-shifts of $u$.
\end{defn}

\begin{prop}\label{prop:D set quadratic size}
    Let $u = (u_{1}y^{\alpha_{1}}, y^{\alpha_{2}}) \in G_{\mathrm{mod}}$ be $\phi$-CR, and let $q$ denote the length of $u_{1}$. Let $\overline{\mathcal{D}_{u}}$ denote the finite set of minimal length representatives from Step 4 of our algorithm. Then $n \leq |\overline{\mathcal{D}_{u}}| \leq 2n^{2}q$, where $n$ denotes the rank of the free group given by the group presentation as in \cref{eq: even semidirect pres}.
\end{prop}
\begin{proof}
    Let $Y_{u}$ denote the number of $y$-shifts from $u$ to obtain an element $v = (v_{1}, v_{2}) \in G_{\mathrm{mod}}$, such that $v_{1} = u_{1}y^{\alpha_{1}}$. Similarly let $X_{u}$ denote the number of BF $x$-shifts from $u$ to obtain an element $w \in (w_{1}, w_{2}) \in G_{\mathrm{mod}}$, such that $w_{1} = u_{1}y^{\alpha_{1}}$. By definition, the size of $\overline{\mathcal{D}_{u}}$ is bounded above by $X_{u}Y_{u}$. For any $y$-shifts $v \in G_{\mathrm{mod}}$ obtained from $u$, $v$ is of the form $v =_{G} \Phi_{-\lambda}(u_{1})y^{\alpha_{1}+t(1-\e_{y})}y^{\alpha_{2}}$, for some $\lambda \in \Z$. Here $v$ has the same quotient element as $u$ precisely when $t=0 \Mod{n}$, and so $Y_{u} = n$. We now establish an upper bound for $X_{u}$.

    We consider what happens when we apply a full $x$-shift to $u$ (recall this is equivalent to $q$ BF $x$-shifts of $u$). We want to determine an upper bound on the number of full $x$-shifts required to obtain an element with the same quotient element as $u$. We note this process can be excessive, in that we may return to the same quotient element as $u$ via a smaller number of BF $x$-shifts of subwords of $u_{1}$. However, considering full $x$-shifts of $u$ will give us an upper bound. 
    
    Let $k(u)$ denote the $k$-th full $x$-shift of $u$. We consider different cases for $\e_{x}$ and $\e_{y}$. Let $\sigma$ denote the exponent sum of $u_{1}$. First suppose $\e_{y}=1$ and $\e_{x} = -1$. We calculate the $k$-th full $x$-shift of $u$ as follows:
    \begingroup
    \addtolength{\jot}{0.5em}
    \begin{align*}
        u &=_{G} u_{1}y^{\alpha_{1}}y^{\alpha_{2}}, \\
        1(u) &=_{G} \Phi_{-\alpha_{1}}\left(\left[\phi(u_{1})\right]_{F}\right) y^{\alpha_{1}+\sigma d}y^{\alpha_{2}}, \\
        2(u) &=_{G} \Phi_{-2\alpha_{1}-\sigma d}\left(\left[\phi^{2}(u_{1})\right]_{F}\right)y^{\alpha_{1}}y^{\alpha_{2}}, \\
        3(u) &=_{G} \Phi_{-3\alpha_{1}-\sigma d}\left(\left[\phi^{3}(u_{1})\right]_{F}\right)y^{\alpha_{1}+\sigma d}y^{\alpha_{2}}, \\
        & \quad \vdots
    \end{align*}
    \endgroup
    This leads to a more general result: 
    \begingroup
    \addtolength{\jot}{0.5em}
    \begin{align*}
         (2k)(u) &=_{G} \Phi_{-2k\alpha_{1}-k\sigma d}\left(\left[\phi^{2k}(u_{1})\right]_{F}\right)y^{\alpha_{1}}y^{\alpha_{2}}, \\
         (2k+1)(u) &=_{G} \Phi_{-(2k+1)\alpha_{1}-k\sigma d}\left(\left[\phi^{2k+1}(u_{1})\right]_{F}\right)y^{\alpha_{1}+\sigma d}y^{\alpha_{2}}.
    \end{align*}
    \endgroup
    We note that in order to guarantee our $y$-power of the quotient element returns to $y^{\alpha_{1}}$, then we need to take an even number of full $x$-shifts. For the free component, we first consider the change of indices: this equals $\Phi_{t}$, where $t \in \{-2k\alpha_{1}-k\sigma d, -(2k+1)\alpha_{1}-k\sigma d \}$. As we need to take an even number of full $x$-shifts, we only need to consider the case where $t = -2k\alpha_{1}-k\sigma d$. If $k = n$, then $t =0 \Mod{n}$ as required, and so we need a minimum of $2n$ full $x$-shifts. 

    Since $\e_{x} = -1$, we can use the following relations from \cref{lemma:free component complexity 2}:
    \begingroup
    \addtolength{\jot}{0.5em}
    \begin{align*}
       \left[\phi^{2k}(u_{1})\right]_{F} &=_{G} \phi^{2k}_{F}\left(x^{p_{1}}_{i_{1}}\right)\phi^{2k}_{F}\left(x^{p_{2}}_{i_{2}}\right) \dots \phi^{2k}_{F}\left(x^{p_{q}}_{i_{q}}\right), \\
        \left[\phi^{2k+1}(u_{1})\right]_{F} &=_{G}  \phi^{2k+1}_{F}\left(x^{p_{1}}_{i_{1}}\right)\Phi_{p_{1}d}\left(\phi^{2k+1}_{F}\left(x^{p_{2}}_{i_{2}}\right)\right)\dots \Phi_{(p_{1}+\dots + p_{q-1})d}\left(\phi^{2k+1}_{F}\left(x^{p_{q}}_{i_{q}}\right)\right).
    \end{align*}
    \endgroup
    After an even number of full $x$-shifts, we will return to $u_{1}$ precisely when $\phi^{2k}_{F}(x^{p_{i}}_{i}) = x^{p_{i}}_{i}$, for all $1 \leq i \leq q$. Recall from \cref{lemma:powers of phi free gens} that $\phi^{2k}(x^{p_{i}}_{i}) = x^{p_{i}}_{[i+kd]}$, which equals $x^{p_{i}}_{i}$ when $k = 0 \Mod{n}$, and so $2n$ full $x$-shifts is sufficient to return to $u_{1}$. Therefore $X_{u} \leq 2nq$. 

    Next suppose $\e_{y} = \e_{x} = 1$. Each $k$-th full $x$-shift of $u$ has the form
    \begingroup
    \addtolength{\jot}{0.5em}
    \begin{align*}
        u &=_{G} u_{1}y^{\alpha_{1}}y^{\alpha_{2}}, \\
        1(u) &=_{G} \Phi_{-\alpha_{1}}\left(\left[\phi(u_{1})\right]_{F}\right) y^{\alpha_{1}+\sigma d}y^{\alpha_{2}}, \\
        2(u) &=_{G} \Phi_{-2\alpha_{1}-\sigma d}\left(\left[\phi^{2}(u_{1})\right]_{F}\right)y^{\alpha_{1}+2\sigma d}y^{\alpha_{2}}, \\
        3(u) &=_{G} \Phi_{-3\alpha_{1}-3\sigma d}\left(\left[\phi^{3}(u_{1})\right]_{F}\right)y^{\alpha_{1}+3\sigma d}y^{\alpha_{2}}, \\
        &\quad \vdots
    \end{align*}
    \endgroup
    This gives a more general result:
    \[ k(u) =_{G} \Phi_{-k\alpha_{1}-\frac{k(k-1)}{2}\sigma d}\left(\left[\phi^{k}(u_{1})\right]_{F}\right)y^{\alpha_{1}+k\sigma d}y^{\alpha_{2}}.
    \]
    Here our $y$ exponent of the quotient element returns to $\alpha_{1}$ after $n$ full $x$-shifts. The change of indices equals $\Phi_{t}$, where $t = -k\alpha_{1}-\frac{k(k-1)}{2}\sigma d$. Again $t=0 \Mod{n}$ when $k = 2n$, and so we need $2n$ full $x$-shifts. For the free component, we have (by \cref{lemma:free component complexity 2})
    \[ [\phi^{k}(u_{1})]_{F} = \phi^{k}_{F}\left(x^{p_{1}}_{i_{1}}\right)\Phi_{kdp_{1}}\left(\phi^{k}_{F}\left(x^{p_{2}}_{i_{2}}\right)\right)\dots \Phi_{kd(p_{1}+\dots + p_{1-1})}\left(\phi^{k}_{F}\left(x^{p_{q}}_{i_{q}}\right)\right).
    \]
    Each $\Phi_{kd(p_{1}+\dots + p_{i})}$ term will leave the indices unchanged when $k = n$. For the generators, we have
    \[ \phi^{2k}_{F}(x^{p_{i}}_{i}) = 
    \begin{cases}
        x^{p_{i}}_{i}, & p_{i} = 1, \\
        x^{p_{i}}_{[i+2kd]}, & p_{i} = -1.
    \end{cases}
    \]
    by \cref{lemma:powers of phi free gens}. Again we need a minimum of $2n$ full $x$-shifts, and so $X_{u} \leq 2nq$. Finally, when $\e_{y} = -1$, the formulae for $k(u)$ remain the same, and we can apply a similar proof, using \cref{lemma:powers of phi free gens} and \cref{lemma:free component complexity 2}, to show $X_{u} \leq 2nq$ for all cases. This gives us our upper bound for $|\overline{\mathcal{D}_{u}}|$.

    For the lower bound, suppose $\overline{\mathcal{D}_{u}}$ consists only of elements obtained from $u$ via $y$-shifts only. That is, any BF $x$-shifts generate elements already found by $y$-shifts. We need a minimum of $n$ $y$-shifts, and so $|\overline{\mathcal{D}_{u}}| \geq n$. 
\end{proof}
When we consider certain conditions on $\phi \in \mathrm{Out}(G(m))$ and our input word $u \in G_{\mathrm{mod}}$, we find that this upper bound is indeed excessive. This can be seen from our previous example.
\begin{exmp}
    Recall \cref{exmp:infinite set}, where $|\overline{\mathcal{D}_{u}}| = 18$. By \cref{prop:D set quadratic size}, $|\overline{\mathcal{D}_{u}}| \leq 36$. However, we can prove the lower value directly, based on our input word $u \in G_{\mathrm{mod}}$ and our map $\phi$.

    Since $n=3$, then $Y_{u} = 3$. For $X_{u}$, note that $3\sigma d = 0 \Mod{n}$ and $-3\alpha_{1} = 0 \Mod{n}$. Also, $3d = 0 \Mod{n}$, and so we only need to compute 3 full $x$-shifts of $u$, to obtain an element with the same quotient element as $u$. Since the length of the free component of $u$ is 2, this gives us $X_{u} = 3 \times 2 = 6$. Therefore $|\overline{\mathcal{D}_{u}}| = 18$. 
\end{exmp}
This analysis allows us to determine the complexity of our TCP algorithm.

\Complex

\begin{proof}
    Linear complexity for odd cases was already shown in \cite{crowe_twisted_2024}, so it remains to consider our algorithm for even dihedral Artin groups.
    
    At Step 1, we need to write our input words in modular normal form. First, we rewrite our input words as geodesic normal forms as follows. Suppose $u = u_{1}\dots u_{q}$, where for all $1 \leq i \leq q$,
    \[ u_{i} = 
    \begin{cases}
        x^{r_{i}}_{i}, & r_{i} = \pm 1, i \in \{0, \dots, n-1\}, \\
        y^{r_{i}}, & r_{i} = \pm 1.
    \end{cases}
    \]
    Reading left to right through $u$, we apply the following rules:
    \begin{enumerate}
        \item[(i)] If $u_{i} = x^{r_{i}}_{i}, u_{i+1} = x^{-r_{i}}_{i}$, then freely cancel $u_{i}u_{i+1}$. Similarly cancel any terms of the form $y^{r_{i}}y^{-r_{i}}$. 
        \item[(ii)] If $u_{i} = y^{r_{i}}, u_{i+1} = x^{r_{i+1}}_{j}$, then rewrite $u_{i}u_{i+1}$ as $x^{r_{i+1}}_{[j-r_{i}]}y^{r_{i}}$.
    \end{enumerate}
    This gives us geodesic normal forms in linear time, by moving all $y$-terms from left to right in the word, and freely cancelling when possible. From here, a modular normal form can be obtained again in linear time, by rewriting our $y$ exponent. Step 2 of our algorithm is also linear.

    For Step 3, applying $x$-shifts until our free component satisfies the conditions from \cref{prop:nonlen p cyc reduced} is equivalent to removing the first and last generators of our free component, and adjusting indices and powers where necessary. This again takes linear time. 
    
    Finally by \cref{prop:D set quadratic size}, the size of $\overline{\mathcal{D}_{u}}$ is linear with respect to the length of $u$. To generate all words in $\overline{\mathcal{D}_{u}}$ therefore takes $|\overline{\mathcal{D}_{u}}|\cdot (q + C)$ time, where $C$ is a constant. This is due to the fact that each word we need to generate has length $q + C$, where $C$ is the length of the $y$-exponents, which is bounded by the twisted shift. To then test all possible $\phi \in \mathrm{Out}(G(m))$, we need to compute $\overline{\mathcal{D}_{u}}$ for all $4n$ cases (see \cref{rmk:justify general TCP DAm}), which gives us an overall complexity of quadratic time.  
\end{proof}

\section{Conjugacy problem for extensions of $G(m)$}\label{sec:extension}
For this section we will return to using our geodesic normal form, as defined in \cref{sec:semidirect pres}. The aim of this section is to prove the following.
\orbit
\begin{defn}
    Let $A \leq \mathrm{Aut}(G)$ for a group $G$. The \emph{orbit decidability problem} for $A$, denoted OD(A), takes as input two elements $u,v \in G$, and decides whether there exists $\phi \in A$ such that $v \sim \phi(u)$.  
\end{defn}
Using results from \cref{sec:prelims}, we can prove orbit decidability for all outer automorphisms of $G(m)$.
\begin{prop}\label{prop: even orbit dec}
    Let $m \in \Z_{\geq 2}$ be even and consider the outer automorphism of the form in \cref{even:all auto forms}, i.e.
    \[ \phi \colon x_{0} \mapsto x^{\e_{x}}_{0}y^{d}, \; y \mapsto y^{\e_{y}},
    \]
    where $\e_{x}, \e_{y} \in \{\pm 1\}, \; d \in \Z$. Let $u,v \in F_{n} \rtimes \Z$ be geodesic normal forms. Then it is decidable to determine if $v \sim \phi(u)$.
\end{prop}

\begin{proof}
    Let $u = u_{1}y^{\alpha}$, where $u_{1} = x^{p_{1}}_{i_{1}}\dots x^{p_{q}}_{i_{q}} \in F_{n}$, and let $\sigma$ denote the exponent sum of $u_{1}$. We first consider the image $\phi(u)$. We have $\phi(u) =_{G} [\phi(u_{1})]_{F}\cdot y^{\sigma d + \e_{y}\alpha},$ where we recall $[\phi(u_{1})]_{F}$ denotes the free component of $\phi(u_{1})$, after moving all $y$ terms to the right. We now conjugate this element by some geodesic $w = w_{1}y^{\lambda} \in F_{n} \rtimes \Z$. This gives us
    \begin{align*}
        w^{-1}\phi(u)w &= y^{-\lambda}w^{-1}_{1}\cdot [\phi(u_{1})]_{F} y^{\sigma d + \e_{y}\alpha}\cdot w_{1}y^{\lambda} \\
        &=_{G} \Phi_{-\lambda}\left(w^{-1}_{1}\left[\phi(u_{1})\right]_{F}\right)\Phi_{-\lambda+\sigma d + \e_{y}\alpha}(w_{1}) \cdot y^{\sigma d + \e_{y}\alpha}.
    \end{align*}
    We note that conjugation does not affect the exponent of the $y$ term, and so if $v \sim \phi(u)$, where $v = v_{1}y^{\beta} \in F_{n} \rtimes \Z$, then $\beta = \sigma d + \e_{y}\alpha$. Given $u,v$ and $\phi$ of the form in \cref{even:all auto forms}, we know the value of $\alpha,\beta$ and $\sigma$, and $\e_{y} = \pm 1$. Hence we are able to calculate the value of $d$, if it exists. 

    Moreover, once $d$ is known, we can compute the image $[\phi(u_{1})]_{F}$ explicitly using \cref{prop:non-len p results 2}. Finally, deciding if $v \sim \phi(u)$ is equivalent to deciding if $v_{1} \sim_{\psi} [\phi(u_{1})]_{F}$ for some $\psi \in \mathrm{Aut}(F_{n})$ (see \cite[Prop. 4.1]{bogopolski_orbit_2009}). The twisted conjugacy problem is known to be decidable for free groups \cite[Theorem 1.5]{bogopolski_conjugacy_2006}, and so it is decidable whether $v \sim \phi(u)$.
\end{proof}

\begin{proof}[Proof of \cref{thm:orbit decid}]
Let $\varphi_{1}, \dots \varphi_{s} \in \mathrm{Aut}(G(m))$ be given, and consider $A = \langle \varphi_{1}, \dots \varphi_{s} \rangle \leq \mathrm{Aut}(G(m))$. For each $i = 1,\dots, s$, compute $w_{i} \in G(m)$ such that $\varphi_{i} = \gamma_{i}\phi_{i}$, where $\gamma_{i} \in \mathrm{Inn}(G(m))$ and $\phi_{i} \in \mathrm{Out}(G(m))$. Given two elements $u,v \in G(m)$, we want to decide whether $v \sim \varphi_{i}(u)$ for some $\varphi_{i} \in A$. This has already been determined when $m$ is odd (\cite{crowe_twisted_2024}).

Now suppose $m$ is even. By a similar argument as \cite{crowe_twisted_2024}, our problem reduces to deciding if $v$ is conjugate to $\phi(u)$, where $\phi \in \mathrm{Out}(G(m))$ is of the form in \cref{even:all auto forms}. This is decidable by \cref{prop: even orbit dec}. 
\end{proof}
The following theorem is immediate from \cref{thm:main result TCP solvable}, \cref{thm:orbit decid} and \cite[Theorem 3.1]{bogopolski_orbit_2009}.

\extension

\section*{Acknowledgements}
The author would like to thank Laura Ciobanu and Oorna Mitra for helpful discussions.
\bibliography{references}
\bibliographystyle{plain}
\uppercase{\footnotesize{Department of Mathematics, Heriot-Watt University, and the Maxwell Institute for Mathematical Sciences, Edinburgh, EH14 4AS.}}
\par 
\textit{Email address:} \texttt{ggc2000@hw.ac.uk}

\end{document}